\documentclass[a4paper,10pt]{article}
\usepackage{amsmath,amsthm,amssymb,amscd,epsfig,esint, mathrsfs,graphics,color}
\usepackage{verbatim}
\usepackage[english]{babel}
\newtheorem{thm}{Theorem}
\newtheorem{lem}[thm]{Lemma}

\def\cA{\mathcal A}
\def\fA{\mathbf A}
\def\ue{u_\varepsilon}
\def\cB{\mathcal B}
\def\B{\mathbf{B}}
\def\E{\mathbf{E}}
\def\R{\mathbb R}

\def\Z{\mathbb Z}

\def\M{\mathcal M}

\def\diam{\text{diam}}
\def\div{\operatorname{div }}

\def\supp{\operatorname{supp}}

\def\div{\operatorname{div}}

\title{\sc{Fractional differentiability for solutions \\ of nonlinear elliptic equations}
\footnotetext{\hspace{-0.35cm}
2010 \emph{Mathematics Subject Classification}. 35B65, 35J60, 42B37,49N60.
\endgraf
{\it Key words and phrases}.
Nonlinear elliptic equations, local well-posedness, higher order fractional differentiability, Besov spaces
}}
\author{\it A. L. Bais\'on, A. Clop, R. Giova, J. Orobitg, A. Passarelli di Napoli}
\date{}

\numberwithin{equation}{section}

\begin{document}
\maketitle

\begin{abstract}
We study nonlinear elliptic equations  in divergence form
$$\div\cA(x,Du)=\div G.$$
When $\cA$ has linear growth in $Du$, and assuming that $x\mapsto\cA(x,\xi)$ enjoys $B^\alpha_{\frac{n}\alpha, q}$ smoothness, local well-posedness is found in $B^\alpha_{p,q}$  for certain values of $p\in[2,\frac{n}{\alpha})$ and $q\in[1,\infty]$. In the particular case $\cA(x,\xi)=A(x)\xi$, $G=0$ and $A\in B^\alpha_{\frac{n}\alpha,q}$, $1\leq q\leq\infty$, we obtain $Du\in B^\alpha_{p,q}$ for each $p<\frac{n}\alpha$. Our main tool in the proof is a more general result, that holds also if $\cA$ has growth $s-1$ in $Du$, $2\leq s\leq n$, and asserts local well-posedness in $L^q$ for each $q>s$, provided that $x\mapsto\cA(x,\xi)$ satisfies a locally uniform $VMO$ condition.
 \end{abstract}

\section{Introduction}

The main purpose of this paper consists of analyzing the extra \emph{fractional} differentiability of weak solutions of the following nonlinear elliptic equations in divergence form,
\begin{equation}\label{defeq}
\div \cA(x,Du)= \div G  \qquad\qquad \mathrm{in} \,\,\Omega,
\end{equation}
where $\Omega\subset\R^n$ is a domain, $u:\Omega\to \R $, $ G :\Omega\to  \R^{n} $, and $\cA:\Omega\times \R^{n} \to  \R^{n} $ is a {\emph{Carath\'eodory function with linear growth}}. This means that there are constants $\ell,L,\nu>0$ and $0\leq \mu\leq 1$ such that
\begin{enumerate}
\item[{\rm ($\cA1$)}] $\langle \cA(x,\xi)-\cA(x,\eta),\xi-\eta\rangle \geq  |\xi-\eta|^2$,
\item[{\rm ($\cA2$)}] $|\cA(x,\xi)-\cA(x,\eta)|\leq L  |\xi-\eta|$,
\item [{\rm ($\cA3$)}]$|\cA(x,\xi)|\leq \ell (\mu^2+|\xi|^2)^{\frac{1}{2}}$,
\end{enumerate}
for every $\xi,\eta\in \R^n$ and for a.e. $x\in\Omega$.\\
\\
It is clear that no extra differentiability can be expected for solutions, even if $ G $ is smooth, unless some assumption is given on the $x$-dependence of $\cA$. Thus, we wish to find conditions on $\cA$ under which fractional differentiability assumptions on $G$ transfer to $Du$ \emph{with no losses in the order of differentiation}. \\
\\
The regularity theory for elliptic equations goes back to the seminal works by de Giorgi, Nash and Moser on H\"older continuity of weak solutions. Later on, for linear equations, Meyers found the existence of a number $p_0(n,\nu,L,\ell)$ such that a priori $L^p$ estimates for the gradient hold whenever $p_0'<p<p_0$. In both cases, no regularity for the coefficients is needed (other than measurability). Also, both the $C^\alpha$ and the $L^p$ theory have been extended to  nonlinear Carath\'eodory functions $\cA$ not necessarily having linear growth (we refer the interested reader to the monographs \cite{Gia} and \cite{Giusti} for a complete treatment of the subject). If one seeks for higher differentiability results, then extra assumptions are needed on the coefficients. The classical Schauder estimates are a typical example of this fact, and can be used to prove that H\"older regularity on the independent term $G$ transfers to the gradient $ D  u$ in a nice way, provided the dependence $x\mapsto\cA(x,\xi)$ is also H\"older. A very precise and unified description of such phenomenon can be found at Kuusi-Mingione \cite{KM2}.\\
\\
Even though there is an extensive literature on the regularity theory for equations like \eqref{defeq}, recent works in the planar situation, $n=2$, have shown a renovated interest in determining the higher differentiability of solutions in terms of that of the datum and the coefficients. So far, especial attention has been driven to the case of fractional  Sobolev spaces $W^{\alpha, p}$. It turns out that remarkable differences are appreciated, depending on the quantity $\alpha p$:
\begin{itemize}
\item If $\alpha p > 2$, then $G\in W^{\alpha,q}$ implies $Du\in W^{\alpha,q}$ whenever $q\leq p$ (see e.g. references \cite{CMO} and \cite{KM}).
\item If $\alpha p = 2$, then  $G\in W^{\alpha,q}$ implies $Du\in W^{\alpha, q}$ for every $q<p$, but not if $q=p$. The reason is that coefficients in $W^{\alpha,\frac{2}{\alpha}}_{loc}$ do not necessarily imply bounded derivative solutions. Precise results in this direction are given in \cite{CFMOZ} (for $\alpha=1$) or \cite{BCO} (for $0<\alpha<1$).
\item If $\alpha p<2$, then $G\in W^{\alpha, q}$ implies $Du\in W^{\alpha,q}$ for $q<q_0$ where $q_0$ depends on the ellipticity constants, and is such that $q_0<p$. See for instance \cite{CFMOZ} for the case $\alpha=1$, and \cite{CFR} for $0<\alpha<1$.
\end{itemize}
The results mentioned above refer to the planar Beltrami equation, which is equivalent to $\cA(x,\xi)=A(x)\,\xi$ for some $A(x)$ which is symmetric and has determinant $1$. \\
\\
It turns out that similar phenomena  seem to occur in higher dimensions. Indeed, recent developments  for  nonlinear equations suggest that linearity should not be a restriction, as appropriate counterparts hold even if $\cA$ has superlinear growth, see for instance \cite{raffaella}, \cite{APdN1} and \cite{PdN}. In these works, higher differentiability is obtained from a pointwise condition on the partial map $\cA$.  More precisely, for Carath\'eodory functions $\cA$ with linear growth, it is assumed that there exists a non negative function $g\in L^n_{loc }(\Omega)$ such that 
\begin{equation}\label{hajlasz}
|\cA(x,\xi)-\cA(y,\xi)|\leq |x-y|\,(g(x)+g(y))\,(\mu^2+|\xi|^2)^\frac{1}{2},  
\end{equation}
for almost every $x,y\in\Omega$, and every $\xi\in\R^n$. Under this condition, solutions to \eqref{defeq} with $G=0$ are shown in \cite{PdN} to be such that $Du\in W^{1,p}_{loc}$ for every $p<2$. As a first fractional counterpart to this result, instead of \eqref{hajlasz} one can assume that  there is $g\in L^\frac{n}\alpha_{loc}(\Omega)$ such that
\begin{equation}\label{hajlasz2}
|\cA(x,\xi)-\cA(y,\xi)|\leq |x-y|^\alpha\,(g(x)+g(y))\,(\mu^2+|\xi|^2)^\frac{1}{2},
\end{equation}
for almost every $x,y\in\Omega$, and every $\xi\in\R^n$. It turns out that one gets improved regularity for solutions measured in terms of the Besov spaces $B^\alpha_{p,q}$. 

\begin{thm}\label{maintriebel}
Let $0<\alpha<1$. Assume that $\cA$ satisfies $(\cA1),(\cA2),(\cA3)$,  and that \eqref{hajlasz2} holds for some $g\in L^\frac{n}\alpha$. There exists $p_0=p_0(n,\nu,\ell)>2$ such that if $u\in W^{1,2}_{loc}$ is a weak solution of 
$$\div \cA(x, Du)=0$$
then $Du\in B^\alpha_{p,\infty}$, locally, whenever $2\leq p<\min\{\frac{n}{\alpha}, p_0\}$. If $\cA(x,\xi)=A(x)\xi$ then $2\leq p<\frac{n}{\alpha}$ suffices.
\end{thm}

\noindent
See Section \ref{prelim} for the definition of $B^\alpha_{p,q}$ and the meaning of \emph{locally}. Theorem \ref{maintriebel} holds even for some values $p> 2$, that is, different than the natural summability for $\cA$. The reason for this is a non-standard version of the Caccioppoli inequality, see Lemma \ref{nonstandardcaccioppoli} in Section \ref{prelim}. At the same time, Theorem \ref{maintriebel} seems to be in contrast with the results at \cite{PdN}. 
Indeed, condition \eqref{hajlasz} fully describes equations with coefficients in the Sobolev space $W^{1,n}$, that is, the Triebel-Lizorkin space $F^1_{n,2}$ (see \cite{KYZ} for details), and so in \cite{PdN} the Triebel-Lizorkin scale is nicely transferred from coefficients to solutions. Nevertheless, it is worth mentioning here that there is a jump between \eqref{hajlasz} and \eqref{hajlasz2}, since \eqref{hajlasz2} says that $A\in F^\alpha_{\frac{n}\alpha,\infty}$ (see \cite{KYZ} for details). \\
\\
Unfortunately, the method does not seem to extend to the existing counterparts of \eqref{hajlasz2} that characterize coefficients in $F^\alpha_{\frac{n}\alpha,q}$ for finite values of $q$. Somewhat surprisingly, the Besov setting fits better in this context. To be precise, given $0<\alpha<1$ and $1\leq q\leq \infty$ we say that {\emph{$(\cA4)$ is satisfied}} if there exists a sequence of measurable non-negative functions $g_k\in L^\frac{n}{\alpha}(\Omega)$ such that 
$$\sum_k\|g_k\|_{L^\frac{n}\alpha(\Omega)}^q<\infty,$$
and at the same time
$$|\cA(x,\xi)-\cA(y,\xi)|\leq |x-y|^{\alpha}\,\left(g_k(x )+g_k(y)\right)\,(1+|\xi|^2)^\frac{1}{2}$$
for each $\xi\in\R^n$ and almost every $x,y\in\Omega$ such that $2^{-k}\leq|x-y|<2^{-k+1}$. We will shortly write then that $(g_k)_k\in\ell^q(L^\frac{n}\alpha)$. If $\cA(x,\xi)= A(x)\xi$ and $\Omega=\R^n$ then $(\cA4)$ says that the entries of $A$ belong to $B^\alpha_{\frac{n}\alpha,q}$, see \cite[Theorem 1.2]{KYZ}.
Under $(\cA4)$, we can prove the following result. 

\begin{thm}\label{mainBesovhomog}
Let $0<\alpha<1$ and $1\leq q\leq \infty$, and assume that $\cA$ satisfies $(\cA1),(\cA2),(\cA3),(\cA4)$. There exists a number $p_0=p_0(n,\nu,\ell)>2$ such that if $u\in W^{1,2}_{loc}(\Omega)$ is a weak solution of 
$$\div \cA(x, Du)=0$$
then $Du\in B^\alpha_{p,q}$ provided that  $2\leq p<\min\{\frac{n}{\alpha},p_0\}$. If $\cA(x,\xi)=A(x)\xi$ then $2\leq p<\frac{n}\alpha$ suffices.
\end{thm}

\noindent
Theorem \ref{mainBesovhomog} extends Theorem \ref{maintriebel}, because \eqref{hajlasz2} implies $(\cA4)$ (indeed $F^\alpha_{\frac{n}\alpha,\infty}\subset B^\alpha_{\frac{n}\alpha,\infty}$). The situation changes drastically if one looks at the inhomogeneous equation \eqref{defeq}. Difficulties appear mainly with the third index $q$, due to the fact that if $1\leq p<\frac{n}\alpha$ and $p^\ast_\alpha=\frac{np}{n-\alpha p}$ then the embedding $B^\alpha_{p,q}\subset L^{p^\ast_\alpha}$ only holds if $1\leq q\leq p^\ast_\alpha$, and fails otherwise. We obtain the following result. 

\begin{thm}\label{mainBesov}
Let $0<\alpha<1$, and $1\leq q\leq \infty$. Assume that $\cA$ satisfies $(\cA1),(\cA2),(\cA3), (\cA4) $. There exists a number $p_0=p_0(n,\nu,\ell)>2$ such that the implication
$$G\in B^\alpha_{p,q} \hspace{1cm}\Rightarrow\hspace{1cm}Du\in B^\alpha_{p,q}$$
holds locally, provided that  $ \max\{p_0',\frac{nq}{n+\alpha q}\}<p<\min\{\frac{n}{\alpha}, p_0\}$ and $u$, $G$ satisfy \eqref{defeq}.
\end{thm}

\noindent
The above theorem is sharp, in the sense that one cannot expect $Du$ to belong to a Besov space $B^\beta_{p',q'}$ for any $\beta>\alpha$ (as can be seen from the constant coefficient setting). Moreover, our arguments also show that the restriction $p_0'<p<p_0$ can be removed if $\cA$ is linear in the gradient variable. In fact, we have the following

\begin{thm}\label{mainBesovlinear}
Let $0<\alpha<1$ and $1\leq q\leq\infty$, and assume that $\cA(x,\xi)=A(x)\xi$ satisfies $(\cA1),(\cA2),(\cA3)$. Suppose that the entries of $A(x)$ belong to $B^\alpha_{\frac{n}\alpha, q}$. Then the implication
$$G\in B^\alpha_{p,q} \hspace{1cm}\Rightarrow\hspace{1cm}Du\in B^\alpha_{p,q}$$
holds locally, provided that $\max\{1,\frac{nq}{n+\alpha q}\}< p<\frac{n}{\alpha}$ and $u$, $G$ satisfy \eqref{defeq}.
\end{thm}

\noindent
We do not know if Theorems \ref{mainBesovhomog}, \ref{mainBesov} and \ref{mainBesovlinear} remain true in the Triebel-Lizorkin setting, that is, replacing $\ell^q(L^\frac{n}{\alpha})$ by $L^\frac{n}{\alpha}(\ell^q)$ and $B^\alpha_{p,q}$ by $F^\alpha_{p,q}$. \\
\\
Theorems \ref{maintriebel}, \ref{mainBesovhomog}, \ref{mainBesov} and \ref{mainBesovlinear} rely on the basic fact that the Besov spaces $B^\alpha_{\frac{n}\alpha,q}$ and the Triebel-Lizorkin space $F^\alpha_{\frac{n}\alpha,\infty}$ continuously embed into the $VMO$ space of Sarason (e.g. \cite[Prop. 7.12]{Ha}). Linear equations with $VMO$ coefficients are known to have a nice $L^p$ theory (see \cite{I} for $n=2$ or  \cite{IS} for $n\geq2$). A first nonlinear growth counterpart was found in \cite{KZ} for $\cA(x,\xi)=\langle A(x)\xi,\xi\rangle^{s-2}\,A(x)\xi$, $2\leq s\leq n$. Given $s\in[2,n]$, we say that $\fA:\Omega\times \R^n\to\R^n$ is a \emph{Carath\'eodory function of growth $s-1$} if there are constants $L,\ell,\nu>0$ and $0\leq\mu\leq 1$ such that
\begin{enumerate}
\item[{\rm ($\fA1$)}] $\langle \fA(x,\xi)-\fA(x,\eta),\xi-\eta\rangle \geq    \nu(\mu^2+|\xi|^2+|\eta|^2)^{ \frac{s-2}{2}}|\xi-\eta|^2$,
\item[{\rm ($\fA2$)}] $|\fA(x,\xi)-\fA(x,\eta)|\leq L (\mu^2+|\xi|^2+|\eta|^2)^{ \frac{s-2}{2}}|\xi-\eta|$, and
\item [{\rm ($\fA3$)}]$|\fA(x,\xi)|\leq \ell (\mu^2+|\xi|^2)^{\frac{ s-1}{2}}$,
\end{enumerate}
The following result is our main tool for proving Theorems \ref{maintriebel}, \ref{mainBesovhomog}, \ref{mainBesov} and \ref{mainBesovlinear}.

\begin{thm}\label{mainvmo}
Let  $2\leq s\leq n$, and $q>s$. Assume that $(\fA1), (\fA2) , (\fA3)$ hold, and also that $x\mapsto \fA(x,\xi)$ is locally uniformly in $VMO$. If $u$ is a weak solution of 
\begin{equation}\label{defeqs}\div\fA(x,Du)=\div(|G|^{s-2}G)\end{equation}
with $G\in L^q_{loc}$, then $Du\in L^q_{loc}$. Moreover, there exists a constant $\lambda> 1$ such that the Caccioppoli inequality
$$
\fint_{B} |Du|^q \leq C(n,\lambda, \nu,\ell, L, s,q)\left(1+\frac1{|B|^{q/n}}\fint_{\lambda B}|u|^q+\fint_{\lambda B} |G|^q \right)
$$
holds for any ball $B$ such that $\lambda B\subset \Omega$.
\end{thm}

\noindent
See Section \ref{sectvmo} for the precise definition of \emph{locally uniformly VMO}. The proof of Theorem \ref{mainvmo} is inspired by that of \cite{KZ}, although now new technical difficulties arise from the fully nonlinear structure of $\cA$. The result has its own interest, especially for two reasons. First,  $\xi\mapsto\fA(x,\xi)$ is not assumed to be ${\mathcal C}^1$-smooth. Second, the allowed independent terms do not restrict to finite measures. Under these assumptions many interesting bounds on the size and the oscillations of the solutions and gradients are established in \cite{KM} and \cite{KM2}. Unfortunately, and in contrast to the linear situation, this time the lack of self-adjointness is an obstacle to extend the result for values $q\in(1,s)$. \\
\\
The paper is structured as follows. In Section \ref{prelim} we give some preliminaries on Harmonic Analysis. In Section \ref{sectvmo} we prove Theorem \ref{mainvmo}. In Section \ref{sectiontriebel} we prove Theorem \ref{maintriebel} as it is illustrative for proving Theorem \ref{mainBesovhomog} later. In Section \ref{fract} we prove Theorems \ref{mainBesovhomog}, \ref{mainBesov} and Theorem \ref{mainBesovlinear} .\\
\\
\textbf{Acknowledgements}. Antonio Luis Bais\'on, Albert Clop and Joan Orobitg are partially supported by projects MTM2013-44699 (Spanish Ministry of Science), 2014-SGR-75 (Generalitat de Catalunya) and FP7-607647 (MAnET- European Union). Albert Clop is partially supported by the research program ``Ram\'on y Cajal" of the Spanish Ministry of Science. Raffaella Giova and Antonia Passarelli di Napoli are supported  by the Gruppo Nazionale per l'Analisi Matematica, la Probabilit\`a e le loro Applicazioni (GNAMPA) of the Istituto Nazionale di Alta Matematica (INdAM). Raffaella Giova has been partially supported by Project Legge 5/2007 Regione Campania ``Spazi pesati ed applicazioni al calcolo dell variazioni". Antonia Passarelli di Napoli ihas been partially supported by MIUR through the Project
PRIN (2012)``Calcolo delle Variazioni''  .

\section{Notations and Preliminary Results}\label{prelim}

\bigskip

\noindent
In this paper we follow the usual convention and denote by $c$ a
general constant that may vary on different occasions, even within the
same line of estimates.
Relevant dependencies on parameters and special constants will be suitably emphasized using
parentheses or subscripts. The norm we use
on $\mathbb{R}^n$  will be the standard euclidean one and it will be denoted
by $| \cdot |$. In particular, for vectors $\xi$, $\eta \in \mathbb{R}^n$ we write
$\langle \xi , \eta \rangle $ for the usual inner product
of $\xi$ and $\eta$, and $| \xi | := \langle \xi , \xi \rangle^{\frac{1}{2}}$ for the
corresponding euclidean norm. In what follows, $B(x,r)=B_r(x)=\{y\in \R^n:\,\, |y-x|<r\}$ will denote the ball centered at $x$ of radius $r$.
We shall omit the dependence on the center and on the radius when no confusion arises. 

\bigskip

\subsection{Besov-Lipschitz spaces}

Given $h\in \R^n$ and $v:\R^n\to\R$, let us denote $\tau_hv(x)=v(x+h)$ and $\Delta_hv(x)=v(x+h)-v(x)$. As in \cite[ Section 2.5.12]{Tr}, given $0<\alpha<1$ and $1\leq p,q<\infty$, we say that $v$ belongs to the Besov space $B^\alpha_{p,q}(\R^n)$ if $v\in L^p(\R^n)$ and
$$
\|v\|_{B^\alpha_{p,q}(\R^n)}  =\|v\|_{L^p(\R^n)}+[v]_{\dot{B}^\alpha_{p,q}(\R^n)}<\infty$$
where
$$
[v]_{\dot{B}^\alpha_{p,q}(\R^n)}=\left(\int_{\R^n} \left(\int_{\R^n}\frac{|v(x+h)-v(x)|^p}{|h|^{\alpha p}} dx\right)^\frac{q}{p}\frac{dh}{|h|^n}\right)^\frac1q<\infty
$$
Equivalently, we could simply say that $v\in L^p(\R^n)$ and $\frac{\Delta_h v}{|h|^\alpha}  \in L^q\left(\frac{dh}{|h|^n}; L^p(\R^n)\right)$. As usually, if one simply integrates for $h\in B(0,\delta)$ for a fixed $\delta>0$ then an equivalent norm is obtained, because
$$
\left(\int_{\{|h| \geq\delta\}} \left(\int_{\R^n}\frac{|v(x+h)-v(x)|^p}{|h|^{\alpha p}} dx\right)^\frac{q}{p}\frac{dh}{|h|^n}\right)^\frac1q\leq c(n,\alpha, p, q,\delta)\,\|v\|_{L^p(\R^n)}
$$
Similarly, we say that $v\in B^\alpha_{p,\infty}(\R^n)$ if $v\in L^p(\R^n)$ and
$$
[v]_{\dot{B}^\alpha_{p,\infty}(\R^n)} =\sup_{h\in\R^n}   \left(\int_{\R^n}\frac{|v(x+h)-v(x)|^p}{|h|^{\alpha p }} dx\right)^\frac{1}{p} <\infty$$
Again, one can simply take supremum over $|h|\leq \delta$ and obtain an equivalent norm.  By construction, $B^\alpha_{p,q}(\R^n)\subset L^p(\R^n)$. One also has the following version of Sobolev embeddings (a proof can be found at \cite[Prop. 7.12]{Ha}).

\begin{lem}\label{embedding} 
Suppose that $0<\alpha<1$.
\begin{itemize}
\item[(a)] If $1<p<\frac{n}{\alpha}$ and $1\leq q\leq p^\ast_\alpha$ then there is a continuous embedding $B^\alpha_{p,q }(\R^n)\subset L^{p^\ast_\alpha}(\R^n)$. 
\item[(b)] If $p=\frac{n}{\alpha}$ and $1\leq q\leq \infty$ then there is a continuous embedding $B^\alpha_{p,q}(\R^n)\subset BMO(\R^n)$. 
\end{itemize}
\end{lem}

\noindent
Given a domain $\Omega\subset\R^n$ , we say that $v$ belongs to the local Besov space $B^\alpha_{p,q,loc}$ if $\varphi\,v$ belongs to the global Besov space $B^\alpha_{p,q}(\R^n)$ whenever $\varphi$  belongs to the class $\mathcal{C}^\infty_c(\Omega)$ of smooth functions with compact support contained in $\Omega$. The following Lemma is an easy exercise.  

\begin{lem}\label{localbesov2}
A function $v\in L^p_{loc}(\Omega)$ belongs to the local Besov space $B^{\alpha}_{p,q,loc}$ if and only if
$$\left\|\frac{\Delta_h v}{|h|^\alpha}\right\|_{L^q\left(\frac{dh}{|h|^n}; L^p(B)\right)}<\infty$$
for any ball $B\subset 2B\subset\Omega$ with radius $r_B$. Here the measure $\frac{dh}{|h|^n}$ is restricted to the ball $B(0, r_B)$ on the $h$-space.
\end{lem}
\begin{proof}
Let us fix a smooth and compactly supported test function $\varphi$. We have the pointwise identity
$$
\frac{\Delta_h(\varphi v)(x)}{|h|^\alpha}= v(x+h)\, \frac{\Delta_h \varphi(x)}{|h|^\alpha} +\frac{\Delta_hv(x)}{|h|^\alpha}\,\varphi(x).$$
It is clear that 
$$
\left|v(x+h) \, \frac{ \Delta_h \varphi(x) }{|h|^\alpha}\right|\leq |v(x+h)|\,\|\nabla\varphi\|_\infty\,|h|^{1-\alpha}
$$
and therefore one always has $\frac{\Delta_h \varphi}{|h|^\alpha}  \in L^q\left(\frac{dh}{|h|^n}; L^p(\R^n)\right)$. As a consequence, we have the equivalence
$$
\varphi v\in B^\alpha_{p,q}(\R^n)\hspace{1cm}\Longleftrightarrow\hspace{1cm} \frac{\Delta_hv }{|h|^\alpha}\,\varphi\in  L^q\left(\frac{dh}{|h|^n}; L^p(\R^n)\right).
$$
However, it is clear that $\frac{\Delta_hv }{|h|^\alpha}\,\varphi\in  L^q\left(\frac{dh}{|h|^n}; L^p(\R^n) \right)$ for each $\varphi\in C^\infty_c(\Omega)$ if and only if the same happens for every $\varphi=\chi_B$ and every ball $B\subset 2B\subset\Omega$. The claim follows.
\end{proof}

\noindent
As in \cite[Section 2.5.10]{Tr}, we say that a function $v:\R^n\to\R$ belongs to the Triebel-Lizorkin space $F^\alpha_{p,q}(\R^n)$ if $v\in L^p(\R^n)$ and
$$\|v\|_{F^\alpha_{p,q}(\R^n)}=\|v\|_{L^p(\R^n)}+[v]_{\dot{F}_\alpha^{p,q}(\R^n)}<\infty,$$
where
$$
[v]_{\dot{F}_\alpha^{p,q}(\R^n)}=\left(\int_{\R^n}\left(\int_{\R^n}\frac{|v(x+h)-v(x)|^q}{|h|^{n+\alpha q}}\,dh\right)^\frac{p}{q}dx\right)^\frac1p
$$
Equivalently, we could simply say that $v\in L^p(\R^n)$ and $\frac{\Delta_h v}{|h|^\alpha}  \in L^p\left(dx; L^q(\frac{dh}{|h|^n})\right)$.\\
\\
It turns out that Besov-Lipschitz and Triebel-Lizorkin spaces of fractional order $\alpha\in(0,1)$ can be characterized in pointwise terms. Given a measurable function $v:\R^n\to\R$, a \emph{fractional $\alpha$-Hajlasz gradient for $v$} is a sequence $(g_k)_k$ of measurable, non-negative functions $g_k:\R^n\to\R$, together with a null set $N\subset\R^n$, such that the inequality
$$|v(x)-v(y)|\leq |x-y|^\alpha\,(g_k(x)+g_k(y))$$
holds whenever $k\in\Z$ and $x,y\in\R^n\setminus N$ are such that $2^{-k}\leq|x-y|<2^{-k+1}$. We say that $(g_k)\in \ell^q(\Z; L^p(\R^n))$ if 
 $$\|(g_k)_k\|_{\ell^q(L^p)} = \left(\sum_{k\in\Z}\|g_k\|_{L^p(\R^n)}^q\right)^\frac1q<\infty.$$
Similarly, we write $(g_k)\in L^p(\R^n; \ell^q(\Z))$ if
 $$\|(g_k)_k\|_{L^p(\ell^q)} = \left(\int_{\R^n}\|(g_k(x))_k\|_{\ell^q(\Z)}^p\,dx\right)^\frac1p<\infty.$$
The following result was proven in \cite{KYZ}.

\begin{thm}
Let $0<\alpha<1$, $1\leq p<\infty$ and $1\leq q\leq \infty$. Let $v\in L^p(\R^n)$.
\begin{enumerate}
\item One has $v\in B^\alpha_{p,q}(\R^n)$ if and only if there exists a fractional $\alpha$-Hajlasz gradient $(g_k)_k\in \ell^q(\Z; L^p(\R^n))$  for $v$. Moreover,
$$\|v\|_{B^\alpha_{p,q}(\R^n)}\simeq \inf \|(g_k)_k\|_{\ell^q(L^p)}$$
where the infimum runs over all possible fractional $\alpha$-Hajlasz gradients for $v$.
\item One has $v\in F^\alpha_{p,q}(\R^n)$ if and only if there exists a fractional $\alpha$-Hajlasz gradient $(g_k)_k\in L^p(\R^n;\ell^q(\Z))$  for $v$. Moreover,
$$\|v\|_{F^\alpha_{p,q}(\R^n)}\simeq \inf \|(g_k)_k\|_{L^p(\ell^q)}$$
where the infimum runs over all possible fractional $\alpha$-Hajlasz gradients for $v$.
\end{enumerate}
\end{thm}

\subsection{Nonlinear elliptic equations in divergence form}


This section is devoted to recall some fundamentals results of $L^p_{loc}$-theory for solutions of nonlinear elliptic equations in divergence form that we shall use in what follows.  Our first result is a very well known higher integrability property, that we state in a version suitable for our purposes. 

\begin{thm}\label{highint}
Let $s\in[2,n]$, and let $\fA:\Omega\times\R^n\to\R^n$ satisfy ($\fA1$)--($\fA3$). Let $u\in W^{1,s}_{\mathrm{loc}}(\Omega)$ be a local solution of \eqref{defeqs}. If $G\in L^q_{loc}(\Omega)$ for some $q>s$, then there exists an exponent $t$, $s<t<q$ such that $Du\in L^t_{loc}(\Omega)$. Moreover, the following estimate
$$\left(\fint_{B_R}|Du|^t\,dx\right)^{\frac{1}{t}}\le C \left(\fint_{B_{2R}}|Du|^s\,dx\right)^{\frac{1}{s}}+\left(\fint_{B_{2R}}|G|^t\,dx\right)^{\frac{1}{t}},$$
holds for every ball $B_R\subset B_{2R}\Subset \Omega$.
\end{thm}
\noindent
For the proof we refer to \cite[Theorem 6.7, p. 204]{Giusti}. Next, we state a regularity result for solutions of homogenous non linear elliptic equations of the form
$$ \div {\cal{B}}(Du)=0$$
where $ \mathcal{B}:\R^n\to\R^n$ an {\emph{autonomous Carath\'eodory function with growth $s-1$}}. This means that 
\begin{enumerate}
\item[($\cB1$)] $ \langle \cB(\xi)-\cB(\eta),\xi-\eta\rangle \geq    \nu(\mu^2+|\xi|^2+|\eta|^2)^{\frac{s-2}{2}}|\xi-\eta|^2$, 
\item[($\cB2$)] $ |\cB(\xi)-\cB(\eta)|\leq L (\mu^2+|\xi|^2+|\eta|^2)^{\frac{s-2}{2}}|\xi-\eta|$ , and
\item[($\cB3$)] $ |\cB(\xi)|\leq \ell(\mu^2+|\xi|^2)^\frac{s-1}{2}$,
\end{enumerate}
for each $\xi,\eta\in \R^n$.  We recall the following.

\begin{thm}\label{liphold}
Let $\cB:\R^n\to\R^n$ be such that $(\cB1),(\cB2),(\cB3)$ hold, and $v\in W^{1,s}_{loc}(\Omega)$ be a solution of
$$\div\cB(Dv)=0\hspace{1cm} \mathrm{in}\,\, \Omega,$$
Then, for every ball $B\Subset \Omega$, we have
\begin{itemize}
\item $\sup_{x\in\lambda B}|Dv(x)|\leq \frac{C}{\diam(B) (1-\lambda)}\left(\fint_{B}(1+|Dv|^s)\right)^\frac1s$ for all $0<\lambda<1$.
\item $\fint_{\delta B}|Dv-(Dv)_{\delta B}|^s\leq C\,\delta^{\alpha s}\,\fint_B(1+|Dv|^s)$ \quad for all $0<\delta<1$ and some $\alpha>0$.
\end{itemize}\end{thm}

\noindent
For the proof we refer to Sections 8.3 and 8.7 in \cite{Giusti} or, more specifically to formulas (8.104) and (8.106), p.302-303 in \cite{Giusti}. From previous Theorem, one can easily deduce the following.

\begin{lem}\label{autonomousbvp}
Let $\cB:\R^n\to\R^n$ be such that $(\cB1),(\cB2),(\cB3)$ hold. Let $B\Subset\Omega$ be a ball, and let $w\in W^{1, s }_{loc}(\Omega)$. Then the problem
$$
\begin{cases}
\div\cB(Dv)=0&x\in B,\\v=w&x\in\partial B.
\end{cases}
$$
has a unique  solution  $v\in W^{1,s}(B)$. Moreover, one has:
\begin{itemize}
\item $\sup_{x\in\lambda B}|Dv(x)|\leq \frac{C}{\diam(B) (1-\lambda)}\left(\fint_{B}(1+|Dw|^s)\right)^\frac1s$\quad  for all $0<\lambda<1$.
\item $\fint_{\delta B}|Dv-(Dv)_{\lambda B}|^s\leq C\,\delta^{\alpha s}\,\fint_B(1+|Dw|^s)$ \quad for all $0<\delta<1$ and some $\alpha>0$.
\end{itemize}\end{lem}

\noindent
We conclude this section with a very well known iteration Lemma, that   finds an important application in the so called hole-filling method.  Its proof can be found for example in \cite[Lemma 6.1]{Giusti}.

\begin{lem}\label{holf} Let $h:[r, R_{0}]\to \mathbb{R}$ be a non-negative bounded function and $0<\vartheta<1$,
$A, B\ge 0$ and $\beta>0$. Assume that
$$
h(s)\leq \vartheta h(t)+\frac{A}{(t-s)^{\beta}}+B,
$$
for all $r\leq s<t\leq R_{0}$. Then
$$
h(r)\leq \frac{c A}{(R_{0}-r)^{\beta}}+cB ,
$$
where $c=c(\vartheta, \beta)>0$.
\end{lem}

\subsection{Hodge decomposition}

\noindent
The interested reader can check the contents of this section in the monograph \cite{IM}. We recall that for a vector field $F\in L^p(\mathbb{R}^n,\mathbb{R}^n)$, with $1<p<+\infty$, the Poisson equation $$\Delta w=\mathrm{div}F$$
is  solved by a function $w\in W^{1,p}$ whose gradient can be expressed in terms of the Riesz transform as follows
$$Dw=-(\mathbf{R}\otimes \mathbf{R})(F).$$
The tensor product operator $ \mathbf{R}\otimes \mathbf{R}$ is the $n\times n$ matrix whose entries are the second order Riesz transforms $R_j\circ R_k$ $(1\le j,k\le n)$ and therefore the above identity reads as
$$D_jw=-\sum_{k=1}^n R_jR_k F^k,$$
where $F^k$ denotes the $k-th$ component of the vector field $F$.  Setting $\mathbf{E}= -(\mathbf{R}\otimes \mathbf{R})$
and $\mathbf{B}=\mathbf{Id}-\mathbf{E}$
we then have that
$$F=\mathbf{E}(F)+\mathbf{B}(F).$$
By construction, $\mathbf{E}(F)$ is curl free and $\mathbf{B}(F)$ is divergence free.
Standard Calderon-Zygmund theory yields $L^p$ bounds for the operators $\mathbf{E}$ and $\mathbf{B}$, whenever $1<p<+\infty$. However, we will need a more precise estimate, which is contained in the following stability property of the Hodge decomposition.

\begin{lem}\label{hodget}
Let $w\in W^{1,p}(\R^n)$, and let  $1<p<\infty$. Then there exist vector fields $\E\in L^{p'}(\R^n)$ with $\mathrm{curl}(\E)=0$ and $\B\in L^{p'}(\R^n)$ with $\mathrm{div}(\B)=0$ such that
\begin{equation}\label{hodgell}
Dw|Dw|^{p-2}=\E+\B.
\end{equation}
Moreover
\begin{equation}\label{stihodge1l}
\|\E\|_{L^{p'}(\R^n)}\le C\, \|Dw\|^{p-1}_{L^{p}(\R^n)}
\end{equation}
and
\begin{equation}\label{stihodge2}
\|\B\|_{L^{p'}(\R^n)}\le C\,\max\{p-2,p'-2\}\,\|Dw\|^{p-1}_{L^{p}(\R^n)},
\end{equation}
where $C$ is a universal constant.
\end{lem}

\noindent
The proof of previous Lemma is contained in \cite[Theorem 4]{IS}.  The fact that the constant is independent of $n$ and  $p$ can be derived as in \cite[Corollary 3]{kmel}. We use the above Hodge decomposition to prove the following non-standard Caccioppoli inequality, which is well-known for the community and whose proof is included for the reader's convenience. 

\begin{lem}\label{nonstandardcaccioppoli}
Let $\cA$ be such that $(\cA1),(\cA3)$ hold. There exists a number $p_0=p_0(n,\nu,L)>2$ with the following property. If $p\in(p_0', p_0)$ and $u\in W^{1,p}_{loc}(\Omega)$ is a weak solution of \eqref{defeq} for some $G\in L^2_{loc}\cap L^p_{loc}$ then
$$
\int_{B_0}|Du|^p\leq C\left(\mu^p+\frac{1}{|B_0|^{p/n}}\int_{2B_0}|u|^p+\int_{2B_0}|G|^p\right)
$$
for every ball $B_0\subset 2B_0\subset\Omega$.
\end{lem}
\begin{proof}
Let $B_r$ be a ball of radius $r$, such that $B_r\subset 2B_r\subset\Omega$. Choose radii $r<s<t<2r$, and let $\eta\in {\mathcal C}^\infty_c(\Omega)$ be a cut off function such that $\chi_{B_s}\leq\eta\leq \chi_{B_t}$, and $\|\nabla \eta\|_\infty\leq \frac{c}{t-s}$. We apply Lemma \ref{hodget} to $w=\eta u$, so that we can write
$$|Dw|^{p-2}\,Dw = \E+\B$$
with $\E, \B\in L^{p'}(\R^n)$, both supported on $B_t$, $\mathrm{div}(\B)=0$, $\mathrm{curl}(\E)=0$, and moreover
\begin{equation}\label{pbounds}\aligned
\|\E\|_{L^{p'}(B_t)}&\leq C\|Dw\|_{L^{p}(B_t)}^{p-1},\\
\|\B\|_{L^{p'}(B_t)}&\leq C\,\max\{p-2,p'-2\}\,\|Dw\|_{L^{p}(B_t)}^{p-1}.\\
\endaligned
\end{equation} 
From $\mathrm{curl}(\E)=0$ and $1<p'<\infty$ we know that there is $\varphi\in W^{1,p'}_0(B_t)$ such that $\E=D\varphi$.  Now we test \eqref{defeq} with $\varphi$, and obtain
$$
\int_{B_t} \langle\cA(x, Du),Dw\rangle |Dw|^{p-2}=\int_{B_t}\langle\cA(x,Du),\B\rangle+\int_{B_t} \langle G,D\varphi\rangle
$$
whence
$$
\int_{B_t} \langle\cA(x,Du),Du\rangle\,\eta|Dw|^{p-2}=-\int_{B_t}\langle\cA(x,Du),D\eta\rangle\,u\,|Dw|^{p-2}+\int_{B_t}\langle\cA(x,Du),\B\rangle+\int_{B_t} \langle G,D\varphi\rangle.
$$
Using now $(\cA1)$, $(\cA3)$ and the properties of $\eta$, we get
$$
\nu\int_{B_s}   |Du|^p\leq \ell\int_{B_t\setminus B_s}(\mu+|Du|)\,|D\eta|\,|u|\,|Dw|^{p-2}+\ell\int_{B_t}(\mu+|Du|)\,|\B|+\int_{B_t}|G|\,|D\varphi|.
$$
Now, since $w=\eta u$, Young's inequality tells us that
$$
\int_{B_t\setminus B_s}(\mu+|Du|)\,|D\eta|\,|u|\,|Dw|^{p-2}
\leq C(p)\int_{B_t\setminus B_s}|u|^p\,|D\eta|^{p}+C(p)\int_{B_t\setminus B_s}|Du|^p+ C(p)\mu^p\,|2B_r|
$$
Also, by estimate \eqref{stihodge2} and Young's inequality,
$$\aligned
\int_{B_t}(\mu+|Du|)\,|\B|
&\leq \|\mu+|Du|\|_{L^p(B_t)}\,\|\B\|_{L^{p'}(B_t)}\\
&\leq C\,\max\{p-2,p'-2\}\,\|\mu+|Du|\|_{L^p(B_t)}\,\|Dw\|_{L^{p}(B_t)}^{p-1}\\
&\leq C\,2^{p-2} \max\{p-2,p'-2\}\,\|\mu+|Du|\|_{L^p(B_t)}\,(\|u\,D\eta\|_{L^{p}(B_t)}^{p-1}+\| \eta\,Du\|_{L^{p}(B_t)}^{p-1})\\
&\leq C(p)\mu^p\,|2B_r|+C\,2^{p-2} \max\{p-2,p'-2\}\,\|Du\|_{L^p(B_t)}^p+C(p)\,\|u\,D\eta\|_{L^{p}(B_t)}^p \endaligned$$
where $C$ is the universal constant from Lemma \ref{hodget}. Finally, also by \eqref{stihodge1l} and Young's inequality we get
$$\aligned
\int_{B_t}|G|\,|D\varphi|
&\leq \|G\|_{L^p(B_t)}\,\|D\varphi\|_{L^{p'}(B_t)}\\
&\leq \varepsilon \|D\varphi\|_{L^{p'}(B_t)}^{p'}+C(\varepsilon,p)\|G\|_{L^p(B_t)}^p\\
&\leq C\varepsilon \|Dw\|_{L^{p}(B_t)}^{p}+C(\varepsilon,p)\|G\|_{L^p(B_t)}^p\\
&\leq C\,2^{p-1}\varepsilon \|Du\|_{L^{p}(B_t)}^{p}+C\,2^{p-1}\varepsilon \|u\,D\eta\|_{L^{p}(B_t)}^{p}+C(\varepsilon,p)\|G\|_{L^p(B_t)}^p,
\endaligned$$
where $\varepsilon>0$ will be chosen later. Putting this together, 
$$\aligned
\nu\int_{B_s}   |Du|^p
&\leq
C(p,\ell,\varepsilon)\int_{B_t}|u|^p\,|D\eta|^{p}+C(p,\ell)\int_{B_t\setminus B_s}|Du|^p+C(p)\mu^p\,|2B_r|\\
&+
C(\ell)\,2^{p-2}(\max\{p-2,p'-2\}+2\varepsilon)\,\|Du\|_{L^p(B_t)}^p
+C(\varepsilon)\|G\|_{L^p(B_t)}^p.\\
\endaligned$$
Adding $C(p,\ell)\int_{B_s}|Du|^p$ at both sides, and using the properties of $\eta$, we get
$$\aligned
(\nu+C(p,\ell))\int_{B_s}   |Du|^p
&\leq
\frac{C(p,\ell,\varepsilon)}{(t-s)^p}\int_{2B_r}|u|^p+C(p)\mu^p\,|2B_r|\\
&+\bigg(
C(p,\ell)+
C(\ell)\,2^{p-2}(\max\{p-2,p'-2\}+2\varepsilon\bigg)\,\|Du\|_{L^p(B_t)}^p
+C(\varepsilon)\|G\|_{L^p(B_t)}^p.
\endaligned$$
Above, it is clear that one can always attain
$$
C(\ell)\,2^{p-2}( \max\{p-2,p'-2\}+\varepsilon)\,\leq \frac{\nu}2,
$$
if $\varepsilon>0$ is chosen small enough, and if $p$ is chosen close enough to $2$. We write this as $p\in(p_0', p_0)$. At this point we can use the iteration Lemma \ref{holf} to finish the proof.
\end{proof}

\noindent
The number $p_0$ was precisely described in \cite{AIS} when $n=s=2$, and is unknown otherwise.

\subsection{Maximal functions}
\noindent
Let $1\leq s<\infty$, and let $u\in L^s_{loc}(\R^n; \R)$. We will denote 
$$\M_{s}(u)(x)=\sup_{r>0}\left(\fint_{B(x,r)}|u|^s\right)^\frac1s,$$
When $s=1$ this is the classical Hardy-Littlewood maximal operator. We will also denote
$$\M^\sharp_{s}(u)(x)=\sup_{r>0}\left(\fint_{B(x,r)}|u-u_{B(x,r)}|^s\right)^\frac1s$$
$$\M^\sharp_{s,R}(u)(x)=\sup_{0<r<R}\left(\fint_{B(x,r)}|u-u_{B(x,r)}|^s\right)^\frac1s$$
When $s =1$ they go back to the well-known Fefferman-Stein sharp maximal function. These operators are classical tools in harmonic analysis, we refer the interested reader to \cite{Gr, Gr2}.\\
\\
The following lemma is proven in \cite{KZ} for $s=1$. Its proof for $s>1$ follows similarly. 
 
\begin{lem}\label{sharpmax}
Let $1\leq s<q< \infty$, and let $u\in L^s(\R^n)$.
\begin{itemize}
\item[(i)] One has $\|\M_su\|_{L^q(\R^n)}\leq C(n,s,q)\|u\|_{L^q(\R^n)}$.
\item[(ii)] There exists a constant $k_0=k_0(n,s,q)\geq 2$ such that if $u$ is supported on a ball $B(x_0,R)$ then 
$$\|\M^\sharp_{s,k_0 R}u\|_{L^q(B(x_0, k_0 R))}\geq C(n,s,q)\|u\|_{L^q(B(x_0,R))}.$$
\end{itemize}
\end{lem}

\section{$VMO$ coefficients in $\R^n$}\label{sectvmo}

\noindent
In this section, we assume that $n\geq 2$, and that $\fA:\Omega\times  \R^{n} \to  \R^{n}$ is a Caratheodory function  such that assumptions $(\fA1),(\fA2),(\fA3)$ in the introduction are satisfied.  We also require a control on the oscillations, which is described as follows. Given a ball $B\subset\Omega$, let us denote
$$\fA_B(\xi)=\fint_B\fA(x,\xi)\,dx$$
One can easily check that the operator $\fA_B(\xi)$ also satisfies assumption $(\fA1),(\fA2),(\fA3)$. Now set
\begin{equation}\label{maxosc}
V(x,B)=\sup_{\xi\neq 0}\frac{\left|\fA(x,\xi)-\fA_B(\xi)\right|}{(\mu^2+|\xi|^2)^\frac{s-1}2},
\end{equation}
for $x\in\Omega$ and $B\subset\Omega$.  If $\fA$ is given by the weighted ${s}$-laplacian, that is $\fA(x,\xi)=\gamma(x)\,|\xi|^{s-2}\xi$, one obtains
$$V(x,B)=|\gamma(x)-\gamma_B|,$$
where $\gamma_B=\fint_B \gamma(y)dy$, and so any reasonable $VMO$ condition on $\gamma$ requires that the mean value of $V(x,B)$ on $B$ goes to $0$ as $|B|\to 0$.  Our $VMO$ assumption on general Carath\'eodory functions $\fA$ consists of a uniform version of this fact. Namely, we will say that $x\mapsto\fA(x,\xi)$ is \emph{locally uniformly in $VMO$} if for each compact set $K\subset\Omega$ we have that
\begin{equation}\label{vmocond}
\lim_{R\to 0}\sup_{r(B)<R} \sup_{c(B)\in K}\fint_{B}V(x,B)\,dx=0.
\end{equation}
Here $c(B)$ denotes the center of the ball $B$, and $r(B)$ its radius. \\
\\
The main result in this section is an a priori estimate for weak solutions belonging to $W^{1,q}$ for some $q>s$. It is a local nonlinear version of the classical Theorem by Iwaniec and Sbordone \cite{IS}. Our proof relies on arguments similar to those used in \cite{KZ}.
\noindent

\begin{thm}\label{vmolocal}
Assume that $\fA$ satisfies $(\fA1),(\fA2),(\fA3)$ and that it is locally uniformly in $VMO$, and let $q>s$. There exists $\lambda=\lambda(n,s,q)>1$ with the following property. If $x_0\in\Omega$ then there is a number $d_0>0$ (depending on $\nu, \ell, L$, $s,q,n$, $\fA$ and $x_0$) such that if $u\in W^{1,q}(\Omega;\R)$ is such that
\begin{equation}\label{equation}
\div \fA(x,Du)= \div( | G |^{s-2} G )\hspace{1cm}\text{ weakly in }\Omega
\end{equation}
for some $G\in L^q_{loc}(\Omega; \R)$, then the estimate
\begin{equation}\label{localvmobound}
\fint_{\B_0}|Du|^q\leq C \left(\mu^q+\frac1{d^q}\fint_{\lambda\B_0}|u|^q + \fint_{\lambda\B_0}|G|^q\right)
\end{equation}
holds whenever $0<d<d_0$, $\B_0=B(x_0,d)$ and $\lambda\B_0\subset\Omega$.
\end{thm}

\begin{proof}[Proof of Theorem \ref{vmolocal}]
Let $k_0\geq 2$ be the constant in Lemma \ref{sharpmax}.  Let $\delta\in (0,1)$ to be determined later. We are given a ball $\B_0=B(x_0,d)$, such that $\tilde\B_0=(1+\frac2\delta)k_0\B_0\subset\Omega$. 
\\
The first step consists of localizing the problem. This is done by  choosing arbitrary radii $0<\frac{d}{2}<\rho<r<d$, balls $\B_\rho=B(x_0,\rho)$ and $\B_r=B(x_0,r)$, and a cut off function $\eta:\R^n\to\R$ such that $\eta\in C^\infty_c(\R^n)$, $\chi_{\B_\rho }\leq \eta\leq \chi_{\B_r }$ and $\| D \eta\|_\infty\leq\frac{c(n)}{r-\rho}$. Set $w= \eta^{s'}\,u$. Then clearly $w\in W^{1,q}(\R^n)$ has compact support in $\B_r$ and we have that
\begin{eqnarray}\label{localisation}
	\div \fA(x, Dw)&=&\div \big(\fA(x,Dw)-\fA(x,\eta^{s'}Du)\big)+ \div\big(\fA(x,\eta^{s'}Du)-\eta^s \fA(x,Du)\big) \cr\cr
		&+&\div(\eta^s \fA(x,Du))\cr\cr
		&=&\div (\fA(x,Dw)-\fA(x,\eta^{s'}Du))+\div\big(\fA(x,\eta^{s'}Du)-\eta^s \fA(x,Du)\big)\cr\cr
&+& \eta^s \div(\fA(x,Du))+D(\eta^s)\fA(x,Du).
\end{eqnarray}
For each $y\in k_0\B_0$ and each $0<R<\frac{k_0 d}{\delta}$ we set $B_R=B(y,R)$. Then  $B_R\subset (1+\frac1\delta)k_0\B_0\subset \Omega$, and thus the quantity
$$
\fA_{B_R}(\xi)=\fint_{B_R} \fA(x,\xi)\,dx
$$
is well defined. Let $v$ be the unique solution to the following Dirichlet problem
\begin{equation}\label{Dirichlet}
\begin{cases}
\div\fA_{B_R}(Dv)=0&x\in B_R\\v=w&x\in\partial B_R.
\end{cases}	
\end{equation}
Now, we multiply both sides of the  equality  \eqref{localisation} by $v-w$ and, since $v-w$ vanishes outside of $B_R$, we can integrate by parts thus  getting
\begin{eqnarray*}
\int_{B_R}\Big\langle \fA(x, Dw), Dv-Dw\Big\rangle &=&	\int_{B_R}\Big\langle \fA(x,Dw)-\fA(x,\eta^{s'}Du), Dv-Dw\Big\rangle\cr\cr
&+& \int_{B_R}\Big\langle \fA(x,\eta^{s'}Du)-\eta^s \fA(x,Du), Dv-Dw\rangle \cr\cr
&+& \int_{B_R}\Big\langle \fA(x,Du),  D (\eta^s(v-w))\Big\rangle- \int_{B_R}D(\eta^s)\fA(x,Du)(v-w)\cr\cr
&=&\int_{B_R}\Big\langle \fA(x,Dw)-\fA(x,\eta^{s'}Du), Dv-Dw\Big\rangle\cr\cr
&+&  \int_{B_R}\Big\langle \fA(x,\eta^{s'}Du)-\eta^s \fA(x,Du), Dv-Dw\rangle \cr\cr
&+& \int_{B_R}|G|^{s-2}\Big\langle G,  D (\eta^s(v-w))\Big\rangle- \int_{B_R}D(\eta^s)\fA(x,Du)(v-w),
\end{eqnarray*}
where, in the last equality, we used that $u$ is a solution of the equation  \eqref{equation}. On the other hand, since $v$ is a solution of the Dirichlet problem \eqref{Dirichlet}, we also have
$$\int_{B_R} \Big\langle\fA_B(D v)-\fA_B(Dw), Dv-Dw\Big\rangle = \int_{B_R}\Big\langle \fA_B(Dw), Dw-Dv\Big\rangle$$
and so
\begin{eqnarray*}
\int_{B_R} \Big \langle\fA_B(D v)-\fA_B(Dw), Dv-Dw\Big\rangle &=& \int_{B_R} \Big\langle \fA_B(Dw)-\fA(x,Dw), Dw-Dv\Big\rangle\cr\cr
&+&	\int_{B_R}\Big\langle \fA(x,Dw)-\fA(x,\eta^{s'}Du), Dw-Dv\Big\rangle\cr\cr
&+& \int_{B_R}\Big\langle \fA(x,\eta^{s'}Du)-\eta^s \fA(x,Du), Dw-Dv\rangle \cr\cr
&+& \int_{B_R}|G|^{s-2}\Big\langle G,  D (\eta^s(w-v))\Big\rangle- \int_{B_R}D(\eta^s)\fA(x,Du)(w-v)\cr\cr
&\le&\int_{B_R}|\fA_B(Dw)-\fA(x,Dw)|| Dw-Dv|\cr\cr
&+&	\int_{B_R}| \fA(x,Dw)-\fA(x,\eta^{s'}Du)|| Dw-Dv|\cr\cr
&+&  \int_{B_R}| \fA(x,\eta^{s'}Du)-\eta^s \fA(x,Du)|\,| Dw-Dv|  \cr\cr
&+& \int_{B_R}|G|^{s-1}| D (\eta^{s}(w-v))|+ \int_{B_R}| D (\eta^{s})||\fA(x,Du)||w-v|.
\end{eqnarray*}
We write previous inequality as follows
$$ I_0\le I_1+I_2+I_3+I_4+I_5$$
and we estimate $I_j$ separately. Since $s\ge 2$, by virtue of the ellipticity assumption ($\fA1$), we have that
\begin{equation}\label{I0}
\nu\int_{B_R}|Dv-Dw|^s
\leq \nu\,\int_{B_R}|Dv-Dw|^2\,(\mu^2+|Dv|^2+|Dw|^2)^\frac{s-2}{2}\le I_0
\end{equation}
By the definition of $V(x,B)$ in \eqref{maxosc}, thanks to the assumption \eqref{vmocond} and Young's and H\"older's inequalities, we estimate $I_1$ as follows
\begin{eqnarray}
\label{I1}
I_1&\le& \int_{B_R}V(x,B)\,(\mu^2+|Dw|^2)^\frac{s-1}{2}\,|Dw-Dv|\cr\cr
&\leq& \varepsilon\,\int_{B_R}|Dv-Dw|^s+C(\varepsilon, s) \int_B V(x,B)^{s'}\,(\mu^2+|Dw|^2)^\frac{s}{2}\cr\cr
&\leq &\varepsilon\,\int_{B_R}|Dv-Dw|^s+C(\varepsilon, s) \left(\int_{B_R} V(x,B)^\frac{ts'}{t-s}\right)^\frac{t-s}{t}\,\left(\int_{B_R}(\mu^2+|Dw|^2)^\frac{t}{2}\right)^\frac{s}{t}\cr\cr
&\leq& \varepsilon\,\int_{B_R}|Dv-Dw|^s+C(\varepsilon, s,t,\ell) \left(\int_{B_R} V(x,B)\right)^\frac{t-s}{t}\,\left(\int_{B_R}(\mu^2+|Dw|^2)^\frac{t}{2}\right)^\frac{s}{t},	
\end{eqnarray}	
where $t>s$ is the exponent determined in Lemma \ref{highint},  $\varepsilon>0$ is a parameter that will be chosen later and we used that  the function $V(x,B)$, by virtue of assumption ($\fA3$), is bounded in $\Omega$. By assumption ($\fA2$), the definition of $w$, Young's inequality and the properties of $\eta$, we have
\begin{eqnarray}\label{I2}
I_2 &\le & L \int_{B_R}|Dw-\eta^{s'}Du|(\mu^2+|Dw|^2+|\eta^{s'}Du|^2)^{\frac{s-2}{2}}|Dv-Dw|\cr\cr
&=&L \int_{B_R}|D(\eta^{s'})u|(\mu^2+|Dw|^2+|Dw-D(\eta^{s'})u|^2)^{\frac{s-2}{2}}|Dv-Dw|\cr\cr
&\le& c(s,L)\int_{B_R}|D(\eta^{s'})u|(\mu+|Dw|+|D(\eta^{s'})u|)^{s-2}|Dv-Dw|\cr\cr
&\le& c(s,L)\int_{B_R}|D(\eta^{s'})u|^{s-1}|Dv-Dw|+c(s)\int_{B_R}|D(\eta^{s'})u|(\mu+|Dw|)^{s-2}|Dv-Dw|\cr\cr
&\le& \varepsilon \int_{B_R}|Dv-Dw|^s+\sigma\int_{B_R}(\mu+|Dw|)^s + \frac{c(\varepsilon,\sigma,s)}{(r-\rho)^s}\int_{B_R}|u|^{s},
\end{eqnarray}
where $\varepsilon,\sigma>0$ will be chosen later. We now proceed with the estimate of $I_3$. The properties of $\eta$ and Young's inequality yield
\begin{eqnarray}\label{I3}
I_3 &\le &  \int_{B_R\setminus \B_r}|\fA(x,0)||Dv-Dw|+\int_{B_R\cap (\B_r\setminus \B_\rho)}|\fA(x,\eta^{s'}Du)-\eta^s \fA(x,Du)|| Dv-Dw|\cr\cr
&\le & C(\varepsilon)\int_{B_R\setminus \B_r}| \fA(x,0)|^{s'}+C(\varepsilon)\int_{B_R\cap (\B_r\setminus \B_\rho)}|\fA(x,\eta^{s'}Du)- \fA(x,Du)|^{s'}\cr\cr
&+& C(\varepsilon)\int_{B_R\cap (\B_r\setminus \B_\rho)}|\fA(x,Du)- \eta^{s}\fA(x,Du)|^{s'}
+\varepsilon \int_{B_R}|Dv-Dw|^s\cr\cr
&\le & C(\varepsilon,\ell,s) \mu^s\,R^n+C(\varepsilon,L,s)\int_{B_R\cap(\B_r\setminus \B_\rho)}|\eta^{s'}Du- Du|^{s'}(\mu+|\eta^{s'}Du|+|Du|)^{s'(s-2)}\cr\cr
&+& C(\varepsilon,\ell,s)\int_{B_R\cap (\B_r\setminus \B_\rho)}(1-\eta^{s})^{s'}
|Du|^{s}
+\varepsilon \int_{B_R}|Dv-Dw|^s\cr\cr
&\le & C(\varepsilon,\ell,s) \mu^s\,R^n+C(\varepsilon,L,\ell,s)\int_{B_R}|Du|^s\chi_{_{\B_r\setminus \B_\rho}}+\varepsilon\int_{B_R}|Dv-Dw|^s,
\end{eqnarray}
where we also used assumptions ($\fA2$) and  ($\fA3$). Using Young's inequality again and the properties of $\eta$, we have that
\begin{eqnarray}\label{I4}
I_4 &\le & \int_{B_R}\eta^{s}|G|^{s-1}|Dv-Dw|+\int_{B_R}|D(\eta^{s})||G|^{s-1}|v-w|\cr\cr
&\le& \varepsilon \int_{B_R}|Dv-Dw|^s+ c(\varepsilon)\left(\frac{R^s}{(r-\rho)^s}+1\right)\int_{B_R}|G|^s+\varepsilon\int_{B_R}\frac{|v-w|^s}{R^s}\cr\cr
&\le& \varepsilon \int_{B_R}|Dv-Dw|^s+ c(\varepsilon)\left(\frac{d^sk^s_0}{\delta^s(r-\rho)^s}+1\right)\int_{B_R}|G|^s+C(n,s)\varepsilon \int_{B_R}|Dv-Dw|^s,
\end{eqnarray}
where, in the last estimate, we  used Poincar\'e - Wirtinger inequality and the bound $R<\frac{k_0d}{\delta}$.
Finally, by virtue of Young and Poincar\'e - Wirtinger inequalities and again the properties of $\eta$, we estimate
\begin{eqnarray}\label{I5}
I_5 &\le &\ell \int_{B_R}|D(\eta^{s})|(\mu+|Du|)^{s-2}(\mu+|Du|)|v-w|\cr\cr
&\le& \frac{c(n,\ell)}{(r-\rho)}\int_{B_R}\eta^{\frac1{s-1}}(\mu+\eta^{s'}|Du|)^{s-2}(\mu+|Du|)|v-w|\cr\cr
&=&  \frac{c(n,\ell)}{(r-\rho)}\int_{B_R}\eta^{\frac1{s-1}}(\mu+|Dw-D(\eta^{s'})u|)^{s-2}(\mu+|Du|)|v-w|\cr\cr
&\le&  \frac{c(n,\ell,s)}{(r-\rho)}\int_{B_R}|Dw|^{s-2}(\mu+|Du|)|v-w|\cr\cr
&+& \frac{c(n,\ell,s)}{(r-\rho)}\int_{B_R}(\mu+|D(\eta^{s'})u|)^{s-2}(\mu+|Du|)|v-w|\cr\cr
&\le &  \varepsilon \int_{B_R}|Dv-Dw|^s+\sigma\int_{B_R}|Dw|^s+C(n,\varepsilon,\sigma,\ell,s)\frac{R^s}{(r-\rho)^s}\int_{B_R}|Du|^{s} \cr\cr
&&\quad +\frac{C(n,\varepsilon,\sigma,\ell,s)}{(r-\rho)^s}\int_{B_R}|u|^{s}+C\frac{\mu^s\,R^{n+s}}{(r-\rho)^s}.
\end{eqnarray}
Combining estimates \eqref{I0}, \eqref{I1}, \eqref{I2}, \eqref{I3},   \eqref{I4} and \eqref{I5} we conclude that
\begin{eqnarray*}
\nu\int_{B_R}|Dv-Dw|^s &\le &	\varepsilon(5+C(n,s))\,\int_{B_R}|Dv-Dw|^s+2\sigma\int_{B_R}|Dw|^s+ \frac{c}{(r-\rho)^s}\int_{B_R}|u|^{s}\cr\cr
&+&c \left(\int_{B_R} V(x,B)\right)^\frac{t-s}{t}\,\left(\int_{B_R}(\mu+|Dw|^2)^\frac{t}{2}\right)^\frac{s}{t}\cr\cr
&+&c\frac{R^s}{(r-\rho)^s}\int_{B_R}|Du|^{s}+c\frac{d^sk^s_0}{\delta^s(r-\rho)^s}\int_{B_R}|G|^s\cr\cr
&+&  c\int_{B_R}|Du|^s\chi_{_{\B_r\setminus \B_\rho}} +c\,\mu^s\,R^n+C\frac{\mu^s\,R^{n+s}}{(r-\rho)^s},\end{eqnarray*}
where $c=c(\varepsilon,\sigma,s,n,\ell,L)$.
Choosing $\varepsilon=\frac{\nu}{2(5+C(n,s))}$, we can reabsorb the first integral in the right hand side of previous estimate by the left hand side thus obtaining
\begin{eqnarray}\label{finest}
\frac{\nu}{2}\int_{B_R}|Dv-Dw|^s &\le &2\sigma\int_{B_R}|Dw|^s+ \frac{c}{(r-\rho)^s}\int_{B_R}|u|^{s}\cr\cr
&+&c \left(\int_{B_R} V(x,B)\right)^\frac{t-s}{t}\,\left(\int_{B_R}(\mu+|Dw|^2)^\frac{t}{2}\right)^\frac{s}{t}\cr\cr
&+& c\frac{R^s}{(r-\rho)^s}\int_{B_R}|Du|^{s}+c\left(\frac{d^sk^s_0}{\delta^s(r-\rho)^s}+1\right)\int_{B_R}|G|^s\cr\cr
&+&  c\int_{B_R}|Du|^s\chi_{_{\B_r\setminus \B_\rho}}+ c\,\mu^s\,R^n+C\frac{\mu^s\,R^{n+s}}{(r-\rho)^s},
\end{eqnarray}
where $c=c(\nu,\sigma,s,n,\ell,L)$.
Consider the ball $B_{\delta R}=B(y,\delta R)$, and observe that
$$
\fint_{B_{\delta R}}|Dw-(Dw)_{B_{\delta R}}|^s\leq C(s)\fint_{B_{\delta R}}|Dv-(Dv)_{B_{\delta R}}|^s + C(s)\,\delta^{-n}\fint_{B_R}|Dw-Dv|^s.
$$
Now we can estimate the two terms on the right hand side with the help of estimate \eqref{finest} and Lemma \ref{autonomousbvp} as follows
\begin{eqnarray*}
\fint_{B_{\delta R}}|Dw-(Dw)_{B_{\delta R}}|^s
 &\le &(C(s)\delta^{\alpha s}+2\sigma \delta^{-n})\fint_{B_R}|Dw|^s+ \frac{c\delta^{-n}}{(r-\rho)^s}\fint_{B_R}|u|^{s}\cr\cr
&+&c \delta^{-n}\left(\fint_{B_R} V(x,B)\right)^\frac{t-s}{t}\,\left(\fint_{B_R}(\mu+|Dw|^2)^\frac{t}{2}\right)^\frac{s}{t}\cr\cr
&+& c \delta^{-n}\frac{R^s}{(r-\rho)^s}\fint_{B_R}|Du|^{s}+c\delta^{-n}\left(\frac{d^sk^s_0\delta^{-s}}{(r-\rho)^s}+1\right)\fint_{B_R}|G|^s\cr\cr
&+&  c\delta^{-n}\fint_{B_R}|Du|^s\chi_{_{\B_r\setminus \B_\rho}} +c\,\mu^{s}\,\delta^{-n}+ \mu^s\,\delta^{-n-s}\frac{d^sk^s_0}{(r-\rho)^s}.
\end{eqnarray*}
By the classical theory, since $2B_R\subset\Omega$ and $u$ is a local solution, we have that
$$\fint_{B_R}|Du|^s\le \frac{C}{R^s}\fint_{B_{2R}}|u|^s+ C\fint_{B_{2R}}|G|^s$$
and therefore, from $B_{2R}\subset \tilde\B_0$ we conclude that
\begin{eqnarray*}\fint_{B_{\delta R}}|Dw-(Dw)_{B_{\delta R}}|^s &\leq&c (\,\delta^{\alpha s}+\sigma \delta^{-n})\, \fint_{B_{R}}|Dw|^s \,
\cr\cr
&+&c\left(\frac{\,d^s\,k_0^s\,\delta^{-n-s}}{(r-\rho)^s}+\delta^{-n}\right)\fint_{B_{2R}}|G\,\chi_{\tilde\B_0}|^s+c\left(\frac{\,\delta^{-n-s}}{(r-\rho)^s}+\delta^{-n}\right)\fint_{B_{2R}}|u\,\chi_{\tilde\B_0}|^s
\cr\cr
&+&
c\,\delta^{-n}\,\left(\fint_{B_R} V(x,B_R)dx\right)^\frac{t-s}{t}\,\left(\fint_{B_R}(\mu+|Dw|^2)^\frac{t}{2}\right)^\frac{s}{t}\cr\cr
&+&  c\delta^{-n}\int_{B_R}|Du|^s\chi_{_{\B_r\setminus \B_\rho} }+c\,\mu^s\,\delta^{-n}+\mu^s\,\delta^{-n-s}\frac{d^sk^s_0}{(r-\rho)^s}\cr\cr
&\leq& c\,\delta^{\alpha s}\, \fint_{B_R}|Dw|^s +\left(\frac{c\,k_0^s\,d^s\delta^{-n-s}}{(r-\rho)^s}+\delta^{-n}\right)\fint_{B_{2R}}|G\,\chi_{\tilde\B_0}|^s+\left(\frac{c\,\delta^{-n-s}}{(r-\rho)^s}+\delta^{-n}\right)\fint_{B_{2R}}|u\,\chi_{\tilde\B_0}|^s
\cr\cr
&+&
c\,\delta^{-n}\,\left(\fint_{B_R} V(x,B_R)dx\right)^\frac{t-s}{t}\,\left(\fint_{B_R}(\mu+|Dw|^2)^\frac{t}{2}\right)^\frac{s}{t}\cr\cr
&+&  c\delta^{-n}\int_{B_R}|Du|^s\chi_{_{\B_r\setminus \B_\rho}} +c\,\mu^s\,\delta^{-n}+\mu^s\,\delta^{-n-s}\frac{d^sk^s_0}{(r-\rho)^s},
\end{eqnarray*}
where we chose  $\sigma=\delta^{\alpha s+n}$. Now, we take supremum over all possible values $R\in (0,k_0 d/\delta)$, and we get
$$\aligned
\M^\sharp_{s,k_0 d}(Dw)(y)^s &\leq c\,\delta^{\alpha s}\,\M_{s}(Dw)(y)^s+c\,\delta^{-n}
\M_{s}(|Du| \chi_{_{\B_r\setminus \B_\rho}})^s(y)\\
&+c\left(\frac{d^s\,k_0^s\,\delta^{-n-s}}{(r-\rho)^s}+\delta^{-n}\right)\,\M_{s}(G\,\chi_{\tilde\B_0})^s(y)+c\left(\frac{\,\delta^{-n-s}}{(r-\rho)^s}+\delta^{-n}\right)\M_{s}(u\,\chi_{\tilde\B_0})^s(y)\\
&+c\,\delta^{-n}\,\bigg(\M_{t}(Dw)(y)^s\bigg)\,\sup_{0<R<k_0 d/\delta}\left(\fint_{B_R} V(x,B_R)dx\right)^\frac{t-s}{t}\\
&+c\,\mu^s\,\delta^{-n}+\mu^s\,\delta^{-n-s}\frac{d^sk^s_0}{(r-\rho)^s}.
\endaligned$$
Now, we raise to the power $\frac{q}{s}$, and then integrate with respect to $y$ over $k_0\B_0$. We obtain
$$\aligned
\|&\M^\sharp_{s,k_0 d}(Dw)\|_{L^q(k_0\B_0)}^q
\leq C(s,q)\,\delta^{\alpha q}\,\|\M_{s}(Dw)\|_{L^q(\R^n)}^q+c\,\delta^{\frac{-nq}{s}}{\|\M_{s}(|Du|\chi_{_{\B_r\setminus \B_\rho}})\|^q_{L^q(\R^n)}}\\
&+c\,\left(\frac{k_0^q\,d^{q}\delta^{\frac{-nq}{s}-q}}{(r-\rho)^q}+\delta^{\frac{-nq}{s}}\right)\,\|\M_{s}(G\,\chi_{\tilde\B_0})\|_{L^q(\R^n)}^q+c\,\left(\frac{\delta^{\frac{-nq}{s}-q}}{(r-\rho)^q}+\delta^{\frac{-nq}{s}}\right)\|\M_{s}(u\,\chi_{\tilde\B_0})\|_{L^q(\R^n)}^q\\
&+c\,\delta^{\frac{-nq}{s}}\,\left(\|\M_{t }( Dw)\|_{L^q(\R^n)}^q\right)\,\sup_{y\in k_0\B_0}\sup_{0<R<k_0 d/\delta}\left(\fint_{B_R} V(x,B_R)dx\right)^\frac{t-s}{t}\\
&+c\,\mu^q\,|k_0\B_0|\left(\delta^{\frac{-nq}{s}}+\frac{\delta^{\frac{-nq}{s}-q}d^q}{(r-\rho)^s}\right),
\endaligned$$
where $c=c(n,s,q,\ell,L,\nu,\delta)$.
Now we use Lemma \ref{sharpmax} (i) and (ii), and obtain
\begin{equation}\label{penultima}
\aligned
\|Dw&\|_{L^q(\B_0)}^q
\leq C(n,s,q)\,\delta^{\alpha q}\,\| Dw\|_{L^q(\B_0)}^q +c\,\delta^{\frac{-nq}{s}}{\|Du\,\chi_{_{\B_r\setminus \B_\rho}}\|_{L^q(\R^n)}^q}\\
&+c\,\left(\frac{d^{q}\,k_0^q\,\delta^{\frac{-nq}{s}-q}}{(r-\rho)^q}+\delta^{\frac{-nq}{s}}\right)\,\|G\,\chi_{\tilde\B_0}\|_{L^q(\R^n)}^q+c\left(\frac{\delta^{\frac{-nq}{s}-q}}{(r-\rho)^q}+\delta^{\frac{-nq}{s}}\right)\|u\,\chi_{\tilde\B_0}\|_{L^q(\R^n)}^q\\
&+c\,\delta^{\frac{-nq}{s}}\,\left(\|Dw\|_{L^q(\B_0)}^q\right)\,\sup_{y\in k_0\B_0}\sup_{0<R<k_0 d/\delta}\left(\fint_{B_R} V(x,B_R)dx\right)^\frac{t-s}{t}\\
&+c\,\mu^q\,|k_0\B_0|\left(\delta^{\frac{-nq}{s}}+\frac{\delta^{\frac{-nq}{s}-q}d^q}{(r-\rho)^s}\right).
\endaligned
\end{equation}
{Our next aim} consists of inserting the two terms with $Dw$ on the right hand side into the term on the left hand side, by making their coefficients as small as possible. To do this, we first look at the term $C(n,s,q)\,\delta^{\alpha q}\,\| Dw\|_{L^q(\B_0)}^q$. To be absorbed on the left hand side, it suffices to choose $\delta$ such that
$$
C(n,s,q)\,\delta^{\alpha q}= \frac14\quad\Longleftrightarrow\quad \delta=\frac{1}{[4C(n,s,q)]^{\frac{1}{\alpha q}}}.
$$
Note that this choice of $\delta=\delta(n,s,q,\alpha)>0$ is  independent of $d$. Therefore, taking into account that $\delta$ has been fixed, estimate \eqref{penultima} becomes
\begin{equation}\label{penultimab}
\aligned
\|Dw&\|_{L^q(\B_0)}^q
\leq c\,{\|Du\,\chi_{_{\B_r\setminus \B_\rho}}\|_{L^q(\R^n)}^q}\\
&+c\,\left(\frac{k_0^q\,d^{q}}{(r-\rho)^q}+1\right)\,\|G\,\chi_{\tilde\B_0}\|_{L^q(\R^n)}^q+\left(\frac{c}{(r-\rho)^q}+1\right)\|u\,\chi_{\tilde\B_0}\|_{L^q(\R^n)}^q\\
&+\tilde c\,\,\left(\|Dw\|_{L^q(\B_0)}^q\right)\,\sup_{y\in k_0\B_0}\sup_{0<R<k_0 d/\delta}\left(\fint_{B_R} V(x,B_R)dx\right)^\frac{t-s}{t} \\
&+c\,\mu^q\,|k_0\B_0|\frac{d^q}{(r-\rho)^q},
\endaligned
\end{equation}
with constants $c$ and $\tilde c$ depending on $n,s,q,\ell,L,\nu$ but independent of $d$. Now, if $k_0d<\frac{d(x_0,\partial\Omega)}{2}$ then
$$
\sup_{y\in k_0\B_0}\sup_{0<R<k_0 d/\delta}\left(\fint_{B_R} V(x,B_R)dx\right)^\frac{t-s}{t}
\leq
\sup_{y\in B(x_0, \frac{d(x_0,\partial\Omega)}2)}\sup_{0<R<k_0 d/\delta}\left(\fint_{B_R} V(x,B_R)dx\right)^\frac{t-s}{t}
$$
and moreover, from \eqref{vmocond} we have that
$$
\lim_{d\to 0}\sup_{y\in B(x_0, \frac{d(x_0,\partial\Omega)}2)}\sup_{0<R<k_0 d/\delta}\left(\fint_{B_R} V(x,B_R)dx\right)^\frac{t-s}{t}=0.$$
In particular, $d>0$ can be chosen small enough so that
$$
\tilde c\,\,\sup_{y\in k_0\B_0}\sup_{0<R<k_0 d/\delta}\left(\fint_{B_R} V(x,B_R)dx\right)^\frac{t-s}{t}<\frac14.
$$
Note that the chosen value $d$ certainly depends on $d(x_0,\partial\Omega)$, $\fA$, $\nu$, $\ell$, $L$, $s$, $q$ and $t$. Nevertheless, this allows us to insert the remaining term with $Dw$ into the left hand side. One then gets immediately from \eqref{penultima} and our choice of $w$ that
$$
\aligned
{\int_{\B_\rho}|Du|^q}&\leq
2^n\,\int_{\B_r}|Dw|^q
\leq c\int_{\B_r\setminus \B_\rho }|Du|^q+c\left(\frac{d^q}{(r-\rho)^q}+1\right)\,\int_{\tilde\B_0}|G|^q\\
&+c\left(\frac{1}{(r-\rho)^q}+1\right)\int_{\tilde\B_0}|u|^q+c\,\mu^q\,|\tilde\B_0|\left(\frac{d^q}{(r-\rho)^q}+1\right).
\endaligned
$$
Filling the hole, i.e. adding to both sides of previous inequality the quantity
$$c\int_{\B_\rho}|Du|^q $$ 
we obtain
$$
\aligned
{\int_{\B_\rho}|Du|^q}&\leq \vartheta{\int_{\B_r }|Du|^q}+c\left(\frac{d^q}{(r-\rho)^q}+1\right)\,\int_{\tilde\B_0}|G|^q\\
&+c\left(\frac{1}{(r-\rho)^q}+1\right)\int_{\tilde\B_0}|u|^q+c\,\mu^q\,|\tilde\B_0|\left(\frac{d^q}{(r-\rho)^q}+1\right),
\endaligned$$
with $0<\vartheta<1$. Since the above estimate is valid for arbitrary radii $\frac{d}2<\rho<r<d$, by virtue of Lemma \ref{holf}, we conclude that
$$
\fint_{\frac12\B_0}|Du|^q
\leq C(n,\nu,\ell,L,s,q)\,\left(\mu^q+\fint_{\tilde\B_0}|G|^q+\frac{1}{d^q}\fint_{\tilde\B_0}|u|^q\right).$$
Since $\tilde\B_0=(1+\frac2\delta)k_0\B_0$, the claim follows by simply choosing $\lambda=2(1+\frac2\delta)k_0$.
\end{proof}

\noindent
We are now in a position to give the 
\begin{proof}[Proof of Theorem \ref{mainvmo}]

{ Fix a ball $B_R(x_0)\Subset\Omega$ with $0<R<\lambda d_0$ where $\lambda$ and $d_0$ are the ones determined in Theorem \ref{vmolocal}. Moreover fix a smooth kernel $\phi \in C^{\infty}_{c}(B_{1}(0))$ with $\phi \geq 0$ and $\int_{B_{1}(0)}
\! \phi = 1$, let us consider the corresponding family of mollifiers $( \phi_{\varepsilon})_{\epsilon >0}$
and put
 \begin{equation}\label{funcapp}
\fA_\varepsilon(x,\xi):= \fA(\cdot,\xi)\ast \phi_\varepsilon(x)=\int_{B_1}\phi(\omega)\fA(x{\color{blue}-}\varepsilon\omega,\xi)\,d\omega
\end{equation}
and
\begin{equation}
G_\varepsilon=G\ast \phi_\varepsilon	
\end{equation}
 each positive
$\varepsilon< \mathrm{dist}(B_R,\partial \Omega)$.}
One can easily check that the assumptions $(\fA1),(\fA2),(\fA3)$ imply
\begin{enumerate}
\item[{\rm (H1)}]\qquad
$\langle \fA_\varepsilon(x,\xi)-\fA_\varepsilon(x,\eta),\xi-\eta\rangle \ge    \nu (\mu^2+|\xi|^2+|\eta|^2)^{\frac{s-2}{2}}|\eta-\xi|^2$
\item[{\rm (H2)}]
\qquad $|\fA_\varepsilon(x,\xi)-\fA_\varepsilon(x,\eta)|\le L |\xi-\eta |(\mu^2+|\xi|^2+|\eta|^2)^{\frac{s-2}{2}}$
\item[{\rm (H3)}] \qquad$|\fA_\varepsilon(x,\xi)|\le \ell (\mu^2+|\xi|^2)^{\frac{s-1}{2}} $
\end{enumerate}
for almost every $x\in\Omega$ and for all $\xi,\eta \in \mathbb{R}^{n}$. Moreover, setting
$$V_\varepsilon(x,B_R)=\sup_{\xi\neq 0}\frac{\left|\fA_\varepsilon(x,\xi)-\fA_{\varepsilon,B_R}(\xi)\right|}{(\mu^2+|\xi|^2)^\frac{s-1}2}\quad \mathrm{with}\quad  \fA_{\varepsilon,B_R}(\xi)=\fint_{B_R}\fA_\varepsilon(y,\xi)dy $$
since   $x\to \fA_\varepsilon(x,\xi)$ is ${\mathcal C}^\infty$ smooth,     we have that
$$
\lim_{r\to 0}\sup_{r(B)<r} \sup_{c(B)\in B_R}\fint_{B}V_\varepsilon(x,B)\,dx=0\,.\leqno{\rm (H4)}$$
{For further needs we record that, since $\fA_\varepsilon(x,Du)\in L^{\frac{s}{s-1}}(B_R)$, that
\begin{equation}\label{convdf1}
\fA_\varepsilon(x,Du)\to \fA(x,Du)\qquad \mathrm{strongly\,\,in}\,\,\in L^{\frac{s}{s-1}}(B_R)	
\end{equation}
and also that, since $G\in   L^q(B_R)$,
\begin{equation}\label{convdf2}
G_\varepsilon\to G\qquad \mathrm{strongly\,\,in}\,\,\in L^{q}_{\mathrm{loc}}(B_R).\end{equation}}
{Let $u\in W^{1,s}_{\mathrm{loc}}(\Omega)$ be a   solution of the equation \eqref{defeq}  and let  us denote by $\ue\in W^{1,s}(B_R)$  the unique  solution of the Dirichlet problem $$\begin{cases}
\div \fA_\varepsilon(x,D\ue)=	\div (|G_\varepsilon|^{s-1}G_\varepsilon)\qquad \mathrm{in}\,\, B_R\cr
\ue=u\qquad\qquad\qquad\qquad\qquad\qquad\quad	\mathrm{on}\,\, \partial B_R
\end{cases}\leqno{(P_\varepsilon)}$$
By the classical theory, since $x\to \fA_\varepsilon(x,\xi)$ is $C^\infty$ smooth, we have that $Du_\varepsilon\in L^q$, for every $q\ge s$.}

\noindent {Using $\varphi=\ue-u$ as test function in the equation $(P_\varepsilon)$ and in the equation \eqref{defeq}, we have
\begin{eqnarray*}\label{ugua}
&&\int_{B_R}\Big\langle \fA_\varepsilon(x,D\ue), Du-D\ue\Big\rangle dx=\int_{B_R}\Big|G_\varepsilon|^{s-1}\langle G_\varepsilon, Du-D\ue\Big\rangle dx\\
&&\int_{B_R}\Big\langle \fA(x,Du), Du-D\ue\Big\rangle dx=\int_{B_R}|G|^{s-1}\Big\langle G, Du-D\ue\Big\rangle dx	
\end{eqnarray*}}
{Subtracting the second equation from the first one, we obtain
\begin{equation}\label{ugua}
\int_{B_R}\Big\langle \fA_\varepsilon(x,D\ue)-\fA(x,Du), Du-D\ue\Big\rangle dx=\int_{B_R}\Big\langle|G_\varepsilon|^{s-1} G_\varepsilon-|G|^{s-1} G, Du-D\ue\Big\rangle dx	 \end{equation}}
{Inequality (H1) yields
\begin{eqnarray}\label{convforte}
&& \nu\int_{B_R}(\mu^2+|Du|^2+|D\ue|^2)^{\frac{s-2}{2}}|Du-D\ue|^2\,dx\cr\cr
&\le & 	\int_{B_R}\Big\langle \fA_\varepsilon(x,D\ue)-\fA_\varepsilon(x,Du), Du-D\ue\Big\rangle\,dx\cr\cr
&= & 	\int_{B_R}\Big\langle \fA(x,Du)-\fA_\varepsilon(x,Du), Du-D\ue\Big\rangle\,dx\cr\cr
&+& \int_{B_R}\Big\langle|G_\varepsilon|^{s-1} G_\varepsilon-|G|^{s-1} G, Du-D\ue\Big\rangle dx\cr\cr
&\le &\left(\int_{B_R}|\fA(x,Du)-\fA_\varepsilon(x,Du)|^{\frac{s}{s-1}}dx\right)^{\frac{s-1}{s}}\left(\int_{B_R}|Du-D\ue|^s\,dx\right)^{\frac{1}{s}}\cr\cr
&+&\left(\int_{B_R}| G_\varepsilon-G|^{s}dx\right)^{\frac{s-1}{s}}\left(\int_{B_R}|Du-D\ue|^s\,dx\right)^{\frac{1}{s}},
\end{eqnarray}
where we used the equality \eqref{ugua} and H\"older's inequality.}
{Since $s\ge 2$, by well known means, from estimate \eqref{convforte} we deduce
$$ \int_{B_R}|Du-D\ue|^s\,dx\le c\int_{B_R}| \fA(x,Du)-\fA_\varepsilon(x,Du)|^{\frac{s}{s-1}}\,dx+\int_{B_R}| G_\varepsilon-G|^{s}dx.$$
Taking the limit as $\varepsilon\to 0$ in previous inequality and recalling \eqref{convdf1} and \eqref{convdf2}, we deduce that $\ue$ converges strongly to $u$ in $W^{1,s}$. Since the operator $A_\varepsilon$ satisfies estimates (H1)--(H4) and $Du_\varepsilon\in L^q$ for every $q\ge s$, we are legitimate to apply  the a priori estimate of Theorem \ref{vmolocal} to each $u_\varepsilon$ thus getting
\begin{equation}\label{caccioeps}
\int _{B_\rho}\,|Du_\varepsilon|^q \leq C\,\left(\mu^q+\int_{B_{\lambda\rho}}|u_\varepsilon|^q\,+\int_{B_{\lambda\rho}} |G_\varepsilon|^q\,\right)	
\end{equation}
for every $q>s$ and for every positive $\rho$ such that $B_{\lambda\rho}\subset B_R$.}
{ Let us  define the  decreasing sequence of exponents
$$ \begin{cases}
q_0=q\\
q_j=\frac{nq_{j-1}}{n+q_{j-1}}\qquad j\in \mathbb{N}	
\end{cases}$$
Note that, since $q_j\searrow 0$, there exists $h\in \mathbb{N}$ such that $q_h\le s^*$. Chose now $\rho=\rho_h$ so small to have $\lambda^h\rho<R$ and let $r_i=\lambda^i\rho$.}
Since $G\in L^q(B_R)$ we have $G\in L^{q_i}(B_R)$ for every $i\in \mathbb{N}$ and so we can write inequality \eqref{caccioeps} as follows
\begin{eqnarray}\label{ite}
\int_{B_{r_i}}\,|Du_\varepsilon|^{q_i} &\leq & \frac{C_{q_i}}{r_{i+1}^{q_i}}\,\int_{B_{r_{i+1}}}|u_\varepsilon|^{q_i}+C_{q_i}\int_{B_{r_{i+1}}} |G_\varepsilon|^{q_i}+C_{q_i}\,\mu^{q_i}\,|B_{r_{i+1}}|\cr\cr
	&\leq & \frac{C_{q_i}}{r_{i+1}^{q_i}}\,\left(\int_{B_{r_{i+1}}}|u_\varepsilon|^{q_{i+1}}+|Du_\varepsilon|^{q_{i+1}}\right)^{\frac{q_i}{q_{i+1}}}+C_{q_i}\int_{B_{r_{i+1}}} |G_\varepsilon|^{q_i}+C_{q_i}\,\mu^{q_{i}}\,|B_{r_{i+1}}|\cr\cr
	&\leq & \frac{C_{q_i}}{r_{i+1}^{q_i}}
	\left[\int_{B_{r_{i+1}}}|u_\varepsilon|^{q_{i+1}}+\frac{C_{q_{i+1}}}{r_{i+2}^{q_{i+1}}}\int_{B_{r_{i+2}}}|u_\varepsilon|^{q_{i+1}}+C_{q_{i+1}}\int_{B_{r_{i+2}}}|G_\varepsilon|^{q_{i+1}}+C_{q_{i+1}}\,\mu^{q_{i+1}}\,|B_{r_{i+2}}|\right]^{\frac{q_{i}}{q_{i+1}}}\cr\cr
	&+&C_{q_i}\int_{B_{r_{i+1}}} |G_\varepsilon|^{q_i}+C_{q_i}\,\mu^{q_i}\,|B_{r_{i+1}}|\cr\cr
	&\leq & \frac{C_{q_i}C_{q_{i+1}}}{(r_{i+1}r_{i+2})^{q_i}}\left(\int_{B_{r_{i+2}}}|u_\varepsilon|^{q_{i+1}}\right)^{\frac{q_{i}}{q_{i+1}}}+\frac{C_{q_i}C_{q_{i+1}}}{(r_{i+1})^{q_i}}\left(\int_{B_{r_{i+2}}}|G_\varepsilon|^{q_{i+1}}\right)^{\frac{q_{i}}{q_{i+1}}}+C_{q_i}\int_{B_{r_{i+2}}} |G_\varepsilon|^{q_i}\cr\cr
	&+&C_{q_i}\,\mu^{q_i}\,|B_{r_{i+1}}|+C_{q_{i+1}}C_{q_i}\,\mu^{q_i}\,|B_{r_{i+1}}||B_{r_{i+2}}|^{\frac{q_{i}}{q_{i+1}}}
\end{eqnarray}
where we used first Sobolev inequality and again  inequality at \eqref{caccioeps} and finally Young's inequality.
Iterating estimate \eqref{ite}, from $i=0$ to $i=h-1$, we deduce that
$$\int_{B_{\rho}}\,|Du_\varepsilon|^{q} \leq  \tilde C_{h}\,\left(\int_{B_{\lambda^h \rho}}|u_\varepsilon|^{q_{h}}\right)^{\frac{q}{q_{h}}}
+\tilde C_h \int_{B_{R}} |G_\varepsilon|^{q}+	\bar C_{h}\,\mu^q,$$
where $\tilde C_h=\displaystyle{\Pi_{i=0}^{h-1} \frac{C_{q_i}}{r_{i+1}^{q_i}}}$. Since $q_h\le s^*$, by virtue of the strong convergence of $u_\varepsilon$ to $u$ in $W^{1,s}$, we can pass to limit as $\varepsilon\to 0$ in previous estimate to deduce
 $$\int_{B_{\rho}}\,|Du|^{q} \leq  \tilde C_{h}\,\left(\int_{B_{R}}|u|^{q_{h}}\right)^{\frac{q}{q_{h}}}
+\tilde C_h\int_{B_{R}} |G|^{q}+\tilde C_{h}\,\mu^q,$$
i.e. the conclusion.
\end{proof}

\noindent

 
\section{Proof of Theorem \ref{maintriebel}}\label{sectiontriebel}

\noindent
We first prove that if \eqref{hajlasz2} is satisfied then $\cA$ has the locally uniform $VMO$ property \eqref{vmocond}.  

\begin{lem}\label{triebelimpliesvmo}
Let $\cA$ be such that $(\cA1), (\cA2) ,(\cA3)$ hold. Assume that \eqref{hajlasz2} is satisfied. Then $\cA$ is locally uniformly in $VMO$, that is, \eqref{vmocond} holds with $s=2$.
\end{lem} 
\begin{proof}
We have
$$\aligned
\fint_B V(x,B)\,dx
&=\fint_B \sup_{\xi\neq 0}\frac{|\cA(x,\xi)-\cA_B(\xi)|}{(\mu^2+|\xi|^2)^\frac{1}{2}}\,dx\\
&\leq\fint_B \sup_{\xi\neq 0}\fint_B\frac{|\cA(x,\xi)-\cA (y,\xi)|}{(\mu^2+|\xi|^2)^\frac{1}{2}}\,dy\,dx\\
&\leq\fint_B \sup_{\xi\neq 0}\fint_B (g(x)+g(y))\,|x-y|^\alpha\,dy\,dx\\
&=\fint_B \fint_B (g(x)+g(y))\,|x-y|^\alpha\,dy\,dx\\
&\leq\left(\fint_B \fint_B (g(x)+g(y))^\frac{n}\alpha\,dy\,dx\right)^\frac{\alpha}{n}\,\left(\fint_B\fint_B|x-y|^\frac{n\alpha}{n-\alpha}\,dy\,dx\right)^\frac{n-\alpha}{n}\\
&\leq\left(\frac{1}{|B|}\,\int_{B}g^\frac{n}\alpha\right)^\frac{\alpha}{n}\,C(\alpha,n)\,|B|^\frac{\alpha}{n}=C(n,\alpha)\,\int_B g^\frac{n}{\alpha}\\
\endaligned$$
and thus \eqref{vmocond} holds. 
\end{proof}

\begin{proof}[Proof of Theorem \ref{maintriebel}]
Given a test function $\varphi\in\mathcal{C}^\infty_c(\Omega)$ such that $\supp\tau_{-h}\varphi\subset\Omega$, we test the equation
$$\div\cA(x,Du)=0$$
with $\varphi$ and $\tau_{-h}\varphi$, and combine the resulting identities. We have
$$
\int\langle \cA(x+h, Du(x+h))-\cA(x+h, Du),\nabla\varphi\rangle =-\int\langle\cA(x+h, Du(x))-\cA(x,Du(x)),\nabla\varphi\rangle.
$$
Now, by setting 
$$\cA_h(x,\xi)=\frac1{|h|^\alpha}\left(\cA(x+h, |h|^\alpha\,\xi+Du(x))-\cA(x+h, Du)\right)$$
and $v_h=\frac{\Delta_h u}{|h|^\alpha}$, we immediately see that $v_h$ is a weak solution of 
\begin{equation}\label{eqh0}
\div \cA_h(x,Dv_h)= \div G_h
\end{equation}
where
\begin{equation}\label{Gh2}
G_h(x)=- \frac1{|h|^\alpha}\,\left(\cA(x+h,Du(x))-\cA(x, Du(x))\right).
\end{equation}
It is immediate to check that the new $\cA_h$ still satisfies $(\cA1), (\cA2)$ with the same constants of $\cA$. Moreover, $(\cA3)$ is also satisfied by $\cA_h$ but now with $\mu=0$. We also note that 
$$
|G_h(x)|=\left|\frac{\cA(x+h,Du(x))-\cA(x, Du(x))}{|h|^\alpha} \right|\leq (g(x+h)+g(x))\,(\mu^2+|Du(x)|^2)^\frac{1}{2},
$$
Now, we know from Lemma \ref{triebelimpliesvmo} that $\cA$ is locally uniformly in $VMO$, and so Theorem \ref{mainvmo} ensures that $Du\in L^r_{loc}$ for each finite $r>2$. In particular, if $2\leq p<\frac{n}\alpha$ then $Du\in L^{p^\ast_\alpha}_{loc}$ and as a consequence $G_h\in L^p_{loc}$. It then follows that Lemma \ref{nonstandardcaccioppoli} can be applied to \eqref{eqh0} with $\mu=0$ and so there exists $p_0=p_0(n,\nu,\ell)>2$ such that if one further has $2\leq p<p_0$ then
\begin{equation}\label{apriorivh}
\|Dv_h\|_{L^p(B)}\leq C_0\left(\frac{1}{r_B}\|v_h\|_{L^p(2B)}+\|G_h\|_{L^p(2B)}\right)
\end{equation}
for each ball $B$ with radius $r_B$ such that $2B\subset\Omega$. In terms of $u$, this reads as
$$\aligned
\left\|\frac{\Delta_h(Du)}{|h|^\alpha}\right\|_{L^p(B)}
&\leq C_0\left(\frac{1}{r_B}\left\|\frac{\Delta_h u}{|h|^\alpha}\right\|_{L^p(2B)}+\left\|G_h\right\|_{L^p(2B)}\right)\\
&\leq C_0\left(\frac{1}{r_B}\left\|\frac{\Delta_h u}{|h|^\alpha}\right\|_{L^p(2B)}+\left\|g\right\|_{L^\frac{n}\alpha(2B)}\,\|(1+|Du|^2)^\frac12\|_{L^\frac{np}{n-\alpha p}(2B)}\right)
\endaligned$$
and so taking supremum for $|h|<\delta$, $\delta>0$ small enough, 
$$\aligned
\sup_{h}\left\|\frac{\Delta_h(Du)}{|h|^\alpha}\right\|_{L^p(B)}
&\leq C_0\left(\frac{1}{r_B}\sup_h\left\|\frac{\Delta_h u}{|h|^\alpha}\right\|_{L^p(2B)}+\left\|g\right\|_{L^\frac{n}\alpha(2B)}\,\|(1+|Du|^2)^\frac12\|_{L^\frac{np}{n-\alpha p}(2B)}\right)
\endaligned$$
We now use Lemma \ref{localbesov2} to see that the term $\sup_h\left\|\frac{\Delta_h u}{|h|^\alpha}\right\|_{L^p(2B)}$ is finite, since  $u\in W^{1, p}_{loc}$. We then obtain that $Du\in B^\alpha_{p,\infty,loc}$, as claimed. When $\cA$ is linear in the gradient variable, that is $\cA(x,\xi)=A(x)\xi$, one immediately sees that $x\mapsto\cA_h(x,\xi)$ is locally uniformly in $VMO$, and therefore the restriction $p<p_0$ at \eqref{apriorivh} is not needed. 
\end{proof}

\section{Proof of Theorems \ref{mainBesovhomog}, \ref{mainBesov} and \ref{mainBesovlinear}  }\label{fract}

\noindent
We first prove that if $\cA$ satisfies $(\cA1), (\cA2) ,(\cA3), (\cA4)$ then it is locally uniformly in $VMO$. When $\cA$ is linear in the second variable, this comes from Lemma \ref{embedding}. 

\begin{lem}\label{besovimpliesvmo}
Let $\cA$ be such that $(\cA1), (\cA2) ,(\cA3), (\cA4)$ hold. Then $\cA$ is locally uniformly in $VMO$ , that is, \eqref{vmocond} holds with $s=2$.\end{lem}
\begin{proof}
Given a point $x\in\Omega$, let us write $A_k(x)=\{y\in\Omega: 2^{-k}\leq|x-y|<2^{-k+1}\}$. We have
$$\aligned
\fint_B V(x,B)\,dx
&=\fint_B \sup_{\xi\neq 0}\frac{|\cA(x,\xi)-\cA_B(\xi)|}{(\mu^2+|\xi|^2)^\frac{1}{2}}\,dx\\
&\leq\fint_B \sup_{\xi\neq 0}\fint_B\frac{|\cA(x,\xi)-\cA (y,\xi)|}{(\mu^2+|\xi|^2)^\frac{1}{2}}\,dy\,dx\\
&=\fint_B \sup_{\xi\neq 0}\frac{1}{|B|}\sum_k\int_{B\cap A_k(x)}\frac{|\cA(x,\xi)-\cA (y,\xi)|}{(\mu^2+|\xi|^2)^\frac{1}{2}}\,dy\,dx\\
&\leq\frac{1}{|B|^2}\sum_k\int_B \int_{B\cap A_k(x)} |x-y|^\alpha\,(g_k(x)+g_k(y))\,dy\,dx\\
\endaligned$$
The last term above is bounded by
$$
\left(\frac{1}{|B|^2}\sum_k\int_B \int_{B\cap A_k(x)} |x-y|^\frac{n\alpha}{n-\alpha}dy\,dx\right)^\frac{n-\alpha}{n}\,\left(\frac{1}{|B|^2}\sum_k\int_B\int_{B\cap A_k(x)}(g_k(x)+g_k(y))^\frac{n}\alpha dy\,dx\right)^\frac\alpha{n} =I \cdot II
$$
The first sum is very easy to handle, since
$$
I=\left(\frac{1}{|B|^2}\sum_k\int_B \int_{B\cap A_k(x)} |x-y|^\frac{n\alpha}{n-\alpha}dy\,dx\right)^\frac{n-\alpha}{n} \leq C(n,\alpha)\,|B|^\frac\alpha{n}
$$
Concerning the second, we see that
$$\aligned
II &\leq \left(\frac{1}{|B|^2}\sum_k|B\cap A_k(x)|\int_B  g_k(x) ^\frac{n}\alpha dx\right)^\frac\alpha{n} \\
 &\leq 
\left(\frac{1}{|B|^2}\sum_k \left(\int_B  g_k(x) ^\frac{n}\alpha dx\right)^\frac{\alpha q}{n} \right)^{\frac\alpha{n}\frac{n}{\alpha q}}
\left(\frac{1}{|B|^2}\sum_k|B\cap A_k(x)|^\frac{\alpha q}{\alpha q - n} \right)^{\frac\alpha{n}\frac{\alpha q - n}{\alpha q}}\\
&=
\frac{1}{|B|^\frac2q}\,\left(\sum_k\| g_k\|_{ L^\frac{n}\alpha(B))} ^q\right)^\frac1q
\frac{1}{|B|^{2(\frac\alpha{n}-\frac1q)}}\left(\sum_k|B\cap A_k(x)|^\frac{\alpha q}{\alpha q - n} \right)^{\frac\alpha{n}\frac{\alpha q - n}{\alpha q}}\\
&\leq 
\frac{1}{|B|^\frac2q}\,\left(\sum_k\| g_k\|_{ L^\frac{n}\alpha(B))} ^q\right)^\frac1q
\frac{1}{|B|^{2(\frac\alpha{n}-\frac1q)}}\,C(n,\alpha, q)|B|^\frac{\alpha}{n}=C(n,\alpha,q)\,|B|^{-\frac{\alpha}{n}}\,\left(\sum_k\| g_k\|_{ L^\frac{n}\alpha(B))} ^q\right)^\frac1q
\endaligned$$
thus
$$
\fint_B V(x,B)\,dx\leq I\cdot II \leq C(n,\alpha,q)\,\left(\sum_k\| g_k\|_{ L^\frac{n}\alpha(B))} ^q\right)^\frac1q.
$$
In order to get the $VMO$ condition, it just remains to prove that
$$
\lim_{r\to 0} \sup_{x \in K}\left(\sum_k\| g_k\|_{ L^\frac{n}\alpha(B(x ,r)))} ^q\right)^\frac1q =0
$$
on every compact set $K\subset\Omega$. To do this, we fix $r>0$ small enough, and observe that the function $x\mapsto \|g_k\|_{\ell^q(L^\frac{n}\alpha(B(x,r))}$ is continuous on the set $\{x\in\Omega:d(x,\partial\Omega>r)\}$, as a uniformly converging series of continuous functions. As a consequence, there is a point $x_r\in K$ (at least for small enough $r>0$) such that
$$
\sup_{x\in K}\|g_k\|_{\ell^q(L^\frac{n}\alpha(B(x,r)))} =  \|g_k\|_{\ell^q(L^\frac{n}\alpha(B(x_r,r)))}.
$$
Now, from $\|g_k\|_{ L^\frac{n}\alpha(B(x,r))}\leq \|g_k\|_{L^\frac{n}\alpha(B(x_r,r))}$ and this belongs to $\ell^q$, we can use dominated convergence to say that
$$ 
\lim_{r\to 0}\|g_k\|_{\ell^q(L^\frac{n}\alpha(B(x_r,r)))} = \left(\sum_k\lim_{r\to 0}\left(\int_{B(x_r,r)}g_k^\frac{n}{\alpha}\right)^\frac{q\alpha}{n}\right)^\frac1q .$$
Each of the limits on the term on the right hand side are equal to $0$, since the points $x_r$ cannot escape from the compact set $K$ as $r\to 0$. This finishes the proof.
\end{proof}

\noindent
We  now prove Theorem \ref{mainBesov}.

\begin{proof}[Proof of Theorem \ref{mainBesov}]
Given a test function $\varphi\in\mathcal{C}^\infty_c(\Omega)$ such that $\supp\tau_{-h}\varphi\subset\Omega$, we test the equation with $\varphi$ and $\tau_{-h}\varphi$, and combine the resulting identities. We have
$$
\aligned
\int\langle \cA(x+h, Du(x+h))&-\cA(x+h, Du),\nabla\varphi\rangle =\\
 &\int\langle\Delta_hG,\nabla\varphi\rangle-\int\langle\cA(x+h, Du(x))-\cA(x,Du(x)),\nabla\varphi\rangle.\endaligned
$$
Now, by setting 
$$\cA_h(x,\xi)=\frac1{|h|^\alpha}\left(\cA(x+h, |h|^\alpha\,\xi+Du(x))-\cA(x+h, Du)\right)$$
and $v_h=\frac{\Delta_h u}{|h|^\alpha}$, we immediately see that $v_h$ is a weak solution of 
\begin{equation}\label{eqh}
\div \cA_h(x,Dv_h)= \div G_h
\end{equation}
where
\begin{equation}\label{Gh1}
G_h(x)=\frac1{|h|^\alpha}\,\Delta_h G(x)- \frac1{|h|^\alpha}\,\left(\cA(x+h,Du(x))-\cA(x, Du(x))\right)
\end{equation}
As before, $\cA_h$ still satisfies $(\cA1), (\cA2) ,(\cA3)$ with same constants $\nu$, $L$, $\ell$ but now $\mu=0$. We also note that, by virtue of ($\cA4$) and the assumption on $G$, we have $G_h\in L^p_{loc}$ for almost every $h$. Indeed, this is clear for the first term at \eqref{Gh1}, since by assumption $G\in B^\alpha_{p,q,loc}$. On the other hand, ($\cA4$) tells us that 
$$
\left|\frac{\cA(x+h,Du(x))-\cA(x, Du(x))}{|h|^\alpha} \right|\leq (g_k(x+h)+g_k(x))\,(\mu^2+|Du(x)|^2)^\frac{1}{2},\hspace{.5cm}\text{if }2^{-k}\leq |h|<2^{-k+1}.
$$
Above, $g_k\in L^\frac{n}{\alpha}$ by assumption. Also, $(1+|Du(x)|^2)^\frac{1}{2}\in L^{p^\ast_\alpha}_{loc}$. To see this, use Lemma \ref{embedding} with $p<\frac{n}\alpha$ and $q\leq p^\ast_\alpha$ to see that $G\in L^{p^\ast_\alpha}_{loc}$,  and deduce then that $Du\in L^{p^\ast_\alpha}_{loc}$ from Theorem \ref{mainvmo} (if $p^\ast_\alpha\geq2$) or Lemma \ref{nonstandardcaccioppoli} (if $p^\ast_\alpha<2$ we still have $p_0'<p<p^\ast_\alpha$). Hence, we obtain that $G_h\in L^p_{loc}$. \\
\\
We can use now Lemma \ref{nonstandardcaccioppoli} at \eqref{eqh}. If $B$ is a ball with $(2+|h|)B\subset\Omega$,
\begin{equation}\label{caccvh}
\|Dv_h \|_{L^p(B)}\leq C_0\left(\frac{1}{r_B}\,\|v_h \|_{L^p(2B)}+\|G_h \|_{L^p(2B)}\right),\hspace{1cm}p_0'<p<p_0
\end{equation}
where $r_B$ denotes the radius of $B$, $p_0$ is as in Lemma \ref{nonstandardcaccioppoli}, and the constant $C_0=C_0(n,p,\nu,L,s)$ does not depend on $h$. We now write the above inequality in terms of $u$, and then take $L^q$ norm with the measure $\frac{dh}{|h|^n}$ restricted to the ball $B(0,R)$ on the $h$-space. We obtain that
$$
\left\|\frac{\Delta_h Du}{|h|^\alpha} \right\|_{L^q(\frac{dh}{|h|^n};L^p(B))}\leq C_0\left(\frac1{r_B}\,\left\|\frac{\Delta_hu}{|h|^\alpha} \right\|_{L^q(\frac{dh}{|h|^n};L^p(2B) )}+\|G_h \|_{L^q(\frac{dh}{|h|^n};L^p(2B) )}\right).$$
Above, the first term on the right hand side is finite, since $Du\in L^{p^\ast_\alpha}_{loc}$. In order to estimate the last term, we write
$$
\|G_h \|_{L^q(\frac{dh}{|h|^n}; L^p(2B))}
\leq \left\| \frac{\Delta_hG}{|h|^\alpha}\right\|_{L^q(\frac{dh}{|h|^n}; L^p(2B))}+\left\| \frac{\cA(\cdot+h,Du)-\cA(\cdot, Du)}{|h|^\alpha}\right\|_{L^q(\frac{dh}{|h|^n}; L^p(2B))}
$$
Above, the first term on the right hand side is finite, since by assumption $G\in B^\alpha_{p,q,loc}$. Concerning the second term, denote $r_k=2^{-k}\,R$. We write the $L^q$ norm in polar coordinates, so $h\in B(0,R)$ if and only if $h=r\xi$ for some $0\leq r<R$ and some $\xi$ in the unit sphere $S^{n-1}$ on $\R^n$. We denote by $d\sigma(\xi)$ the surface measure on $S^{n-1}$. We bound the last term above by
$$\aligned
 \int_0^R\int_{S^{n-1}}&\left\| \frac{\cA(\cdot+r\xi,Du)-\cA(\cdot, Du)}{r^\alpha} \right\|_{L^p(2B)}^qd\sigma(\xi)\,\frac{dr}{r} \\
&= \sum_{k=0}^\infty \int_{r_{k+1}}^{r_k}\int_{S^{n-1}}\left\| \frac{\cA(\cdot+r\xi,Du)-\cA(\cdot, Du)}{r^\alpha} \right\|_{L^p(2B)}^qd\sigma(\xi)\,\frac{dr}{r} \\
&\leq 2^{-\alpha q} \sum_{k=0}^\infty \int_{r_{k+1}}^{r_k}\int_{S^{n-1}}\left\|  (\tau_{r\xi}g_k + g_k)\,(1+|Du|^2)^\frac12 \right\|_{L^p(2B)}^qd\sigma(\xi)\,\frac{dr}{r} .
\endaligned$$
Now, using again that $Du\in L^{p^\ast_\alpha}_{loc}$,  
$$\aligned
\left\|  (\tau_{r\xi}g_k + g_k)\,(1+|Du|^2)^\frac12\right\|_{L^p(2B)}
\leq \left\|(1+|Du|^2)^\frac12 \right\|_{L^\frac{np}{n-\alpha p}(2B)}\,\left\| (\tau_{r\xi}g_k+g_k) \right\|_{L^\frac{n}{\alpha}(2B)} \endaligned
$$
On the other hand, we note that for each $\xi\in S^{n-1}$ and $r_{k+1}\leq r\leq r_k$ 
$$\aligned
\|(\tau_{r\xi}g_k+g_k)\|_{L^\frac{n}{\alpha}(2B)}
\leq \|g_k\|_{L^\frac{n}\alpha(2B-r_k\xi)} +\|g_k \|_{L^\frac{n}\alpha(2B)} 
\leq 2 \|g_k \|_{L^\frac{n}\alpha(\lambda B)} 
\endaligned$$
where $\lambda =2+\frac{R}{r_B}$. Hence
$$\aligned
\left\|\frac{\cA(\cdot+h,Du)-\cA(\cdot, Du)}{|h|^\alpha} \right\|_{L^q(\frac{dh}{|h|^n}; L^p(2B))}
&\leq C(n,\alpha,q)\, 
\left\|(1+|Du|^2)^\frac12 \right\|_{L^{p^\ast_\alpha}(2B)}\, 
\| \{g_k\}_k\|_{\ell^q(L^\frac{n}\alpha(\lambda B))}\endaligned$$
where $C(n,\alpha,q)= 2^{1-\alpha}\,\log2\,\sigma(S^{n-1})^\frac1q$. Summarizing,
$$\aligned
\frac1{C_0}\left\|\frac{\Delta_h Du}{|h|^\alpha} \right\|_{L^q(\frac{dh}{|h|^n};L^p(2B))}
&\leq \frac1{r_B}\,\left\|\frac{\Delta_hu}{|h|^\alpha} \right\|_{L^q(\frac{dh}{|h|^n};L^p(2B) )}+
\left\|\frac{\Delta_h G}{|h|^\alpha} \right\|_{L^q(\frac{dh}{|h|^n}; L^p(2B))}\\
&+C(n,\alpha,q)\,\|(1+|Du|^2)^\frac12\|_{L^{p^\ast_\alpha}(2B)}\,\|\{ g_k\}_k\|_{\ell^q(L^\frac{n}\alpha(\lambda B))} 
\endaligned$$
Lemma \ref{localbesov2} now guarantees that $Du\in B^\alpha_{p,q,loc}$ and this concludes the proof.
\end{proof}

\noindent
The proofs of Theorems \ref{mainBesovhomog} and \ref{mainBesovlinear} are almost the same. 

\begin{proof}[Proof of Theorem \ref{mainBesovhomog}]
Arguing again as in the proof of Theorem \ref{mainBesov}, the fact that $G=0$ now tells us that $q\leq p^\ast_\alpha$ is not needed to conclude that $G_h\in L^p_{loc}$ for every single $p<\frac{n}\alpha$, due to Theorem \ref{mainvmo}. As a consequence, \eqref{caccvh} holds for every $p<\min\{p_0,\frac{n}{\alpha}\}$. The rest of the proof follows in the same way.
\end{proof}

\begin{proof}[Proof of Theorem \ref{mainBesovlinear}]
Arguing again as in the Proof of Theorem \ref{mainBesov},  the new equation $\cA_h$ is now linear with $VMO$ coefficients, due to the linearity of $\cA(x,\xi)$ as a function of $\xi$. Also, from $\max\{1,\frac{nq}{n+\alpha q}\}< p<\frac{n}\alpha$ we have $q\leq p^\ast_\alpha<\infty$ and so $G\in L^{p^\ast_\alpha}_{loc}$ implies $Du\in L^{p^\ast_\alpha}_{loc}$ by the results at \cite{IS}. Hence, $G_h$ has an $L^p_{loc}$ majorant, and thus $Dv_h\in L^p_{loc}$ again by \cite{IS}, since $p>1$.  In particular, the restriction $p<\min\{p_0,\frac{n}\alpha\}$ can be replaced by $p<\frac{n}{\alpha}$, and the restriction $p>p_0'$ can be replaced by $p>\max\{1,\frac{nq}{n+\alpha q}\}$. The rest of the proof follows similarly.
\end{proof}

A. L. Bais\'on, A. Clop, J. Orobitg\\
Departament de Matem\`atiques\\
Universitat Aut\`onoma de Barcelona\\
08193-Bellaterra (CATALONIA)

R. Giova\\
Dipartimento di Studi Economici e Giuridici\\
Universit\`a degli Studi di Napoli ``Parthenope" \\
Palazzo Pacanowsky- Via Generale Parisi, 13\\
80123 Napoli (Italy)

A. Passarelli di Napoli\\
Dipartimento di Matematica e Appl. ``R.Caccioppoli"\\
Universit\`a degli Studi di Napoli ``Federico II" \\
Via Cintia, 80126 Napoli (Italy)
\end{document}